\documentclass[a4paper,11pt]{article}

\input{styleart.sty}
\date{}

\begin{document}

\author{Chloé Perin and Rizos Sklinos}
\title{Homogeneity in the free group}
\maketitle

\begin{abstract} We show that any non abelian free group $\F$ is strongly $\aleph_0$-homogeneous, i.e. that finite tuples of elements which satisfy the same first-order properties are in the same orbit under $\Aut(\F)$. We give a characterization of elements in finitely generated groups which have the same first-order properties as a primitive element of the free group. We deduce as a consequence that most hyperbolic surface groups are not $\aleph_0$-homogeneous.
\end{abstract}

\section{Introduction}

Since the works of Sela \cite{Sel1}-\cite{Sel6} and Kharlampovich-Myasnikov \cite{KharlampovichMyasnikov} on Tarski's problem, which show that finitely generated free groups of rank at least $2$ all have the same first-order theory, there has been renewed interest in model-theoretic questions about free groups. With Sela's work, it has become clear that techniques of geometric group theory provide extremely effective ways to tackle these questions. This can be seen in subsequent results such as \cite{SelaStability}, where Sela shows that the theory of the free group is stable, or in \cite{PerinElementary}, where elementary subgroups of a free group of finite rank are shown to be exactly its free factors. Moreover, the geometric nature of the tools often allows these results to be generalized to the class of torsion-free hyperbolic groups (see \cite{Sel7}). In this paper, we apply these techniques to study types of elements and homogeneity in torsion-free hyperbolic groups, and more particularly in free and surface groups. 

The type of a tuple $\bar{a}$ in a structure is the set of all first-order formulas it realizes. More formally, let $\mathcal{L}$ be a first-order language, $\mathcal{M}$ an $\mathcal{L}$-structure, and $B$ a subset of $\mathcal{M}$ (for basic definitions of model theory, the reader is referred to \cite{ChangKiesler, MarkerModelTheory}, or to the short survey given in \cite{Chatzidakis}). A $k$-type over $B$ is a consistent set of first-order formulas $\phi(\bar{x})= \phi(x_1, \ldots, x_k)$ with $k$ free variables and parameters from $B$. We say that a $k$-type $p(\bar{x})$ over $B$ is complete if for every formula $\phi(\bar{x})$ over $B$, either $\phi(\bar{x})$ or $\lnot\phi(\bar{x})$ is in $p(\bar{x})$. Now if $\bar{a}$ is a $k$-tuple of elements of $\mathcal{M}$, the set $tp^{\mathcal{M}}(\bar{a}/B)$ of all formulas with $k$ free variables and parameters in $B$ that $\bar{a}$ realizes in $\mathcal{M}$ is a complete $k$-type $p(\bar{x})$. However, a complete type need not be of this form (we then say the type is not realized in $\mathcal{M}$). From now on by a type we mean a complete type. 

An $\mathcal{L}$-structure $\mathcal{M}$ is strongly $\aleph_0$-homogeneous if for any pair of finite tuples $\bar{a}, \bar{a}'$ with $\tp^{\mathcal{M}}(\bar{a})=\tp^{\mathcal{M}}(\bar{a}')$, there is an automorphism of $\mathcal{M}$ sending $\bar{a}$ to $\bar{a}'$. 
For more general definitions, see Section \ref{BasicsHomogeneitySec}. In the rest of this paper, the structures we will consider are groups and $\mathcal{L}=\{\cdot,^{-1},1\}$ is the language of groups.

Nies proved in \cite{NiesF2} that the free group in two generators $\F_2$ is strongly $\aleph_0$-homogeneous, but the question remained opened for free groups of higher rank. As for many reasons $\F_2$ is a special case, it was clear that the techniques used in \cite{NiesF2} do not apply to higher rank free groups.
The main result of this paper is 
\begin{thm} \label{MainResultIntro} Non abelian free groups are strongly $\aleph_0$-homogeneous.
\end{thm} 

In fact, for free groups of finite rank we prove a stronger result: we show that if two tuples have the same type over a set of parameters $B$, there is an automorphism which sends one to the other and fixes $B$ pointwise.

Note that in \cite{SklinosGenericType}, it was shown that free groups of uncountable rank are not 
$\aleph_1$-homogeneous, so our result completes the picture of the homogeneity spectrum of the free groups.

In a stable group, there is a distinguished kind of types called the generic types (see Section \ref{ModelTheorySec}). 
Following Sela's result on the stability of the theory of finitely generated free groups, 
Pillay showed in \cite{PillayForking} that the theory of the free group admits a unique generic type $p_0$, 
and in \cite{PillayGenericity} that the elements realizing it are exactly the primitive elements 
(i.e. elements which belong to a basis). The type of primitive elements is thus of particular interest to model theorists, and much work has been done towards understanding it (see \cite{PillayForking}, \cite{PillayGenericity}, \cite{SklinosGenericType}). 

We generalize Pillay's result by giving a characterization of elements of a finitely generated group $G$ whose type is $p_0$ in $G$ (note that if such elements exist $G$ must be elementarily equivalent to the free group).

\begin{prop}\label{ElementsOfPrimitiveTypeIntro} Let $G$ be a finitely generated group. Let $u$ be an element of $G$. The type of $u$ in $G$ is $p_0$ if and only if $G$ admits a structure of hyperbolic tower over $\langle u \rangle$. 
\end{prop}

Hyperbolic towers are structures defined by Sela, they are fundamental groups of complexes obtained by successive gluing of hyperbolic surfaces along their boundary to a ground floor complex, in such a way that the surfaces retract on the lower levels. For a formal definition, see Section \ref{HypTowersAndPreretractionsSec}. 

An immediate consequence of this characterization is that $p_0$ is not realized in the fundamental group of the connected sum of four projective planes. In other surface groups, we show using this characterization that the type $p_0$ is realized both by elements which represent simple closed curves on the surface and elements which don't. An immediate consequence is the following:

\begin{prop} Let $\Sigma$ be a closed hyperbolic surface which is not the connected sum of three or four projective planes. Then the set of elements of $\pi_1(\Sigma)$ representing simple closed curves on $\Sigma$ is not definable over the empty set.
\end{prop}

We also get: 
\begin{prop} Let $\Sigma$ be a closed hyperbolic surface which is not the connected sum of three or four projective planes. Then $\pi_1(\Sigma)$ is not strongly $\aleph_0$-homogeneous.
\end{prop}

The paper is structured as follows: we start by giving in Section \ref{ModelTheorySec} basic definitions and examples about homogeneity, and some overview of stable groups and of the model theoretic properties of the free groups. We devote Section \ref{ModularGroupAndJSJSec} to the description of tools which are essential in proving strong $\aleph_0$-homogeneity of the finitely generated free groups: the modular group, and JSJ decompositions. In Section \ref{FactorSetSec}, we recall some results about morphisms between torsion-free hyperbolic groups, obtained by the shortening argument of Rips and Sela. We are then able in Section \ref{SpecialCaseSec} to show Theorem \ref{MainResultIntro} in a special case, under some additional assumption on the tuples $\bar{a}$ and $\bar{a}'$. This example is given as a toy case to help understand how the proof works in general, it is not needed in the proof of the main result. The following section describes the structure of hyperbolic tower, and states results which enable us to claim that a group admits such a structure. In Section \ref{MainResultSec}, we prove Theorem \ref{MainResultIntro}. The goal of Section \ref{ElementsOfprimitiveTypeSec} is to prove the characterization of elements of type $p_0$ in a finitely generated group $G$ given by Proposition \ref{ElementsOfPrimitiveTypeIntro}. Finally, in Section \ref{SurfaceCaseSec}, we use this characterization to deduce the non homogeneity of surface groups. 

We wish to thank Anand Pillay for suggesting this problem and for useful advice, and Zlil Sela for many helpful conversations. We are also grateful to Vincent Guirardel for the proof of Lemma \ref{ModGroupStabilizes}, and to Gilbert Levitt for his comments on the preliminary versions of the paper.

\section{Some model theory} \label{ModelTheorySec}

The aim of this section is give the definitions of the various notions of homogeneity, as well as a brief overview of some notions which are of interest to place the results of this paper in their model-theoretic context. 

\subsection{Homogeneity: definitions and examples} \label{BasicsHomogeneitySec}

We now want to give a formal account of the different notions of homogeneity, as well as some examples. Again we refer the reader to \cite{ChangKiesler, MarkerModelTheory}, or to the short survey given in \cite{Chatzidakis} for basic model theory definitions.

\begin{defi} Let $\kappa$ be an infinite cardinal. An $\mathcal{L}$-structure $\mathcal{M}$ is $\kappa-$homogeneous if whenever $\beta<\kappa$ and $\bar{a},\bar{b}$ are $\beta$-tuples of elements of $\mathcal{M}$ with $tp^{\mathcal{M}}(\bar{a})=tp^{\mathcal{M}}(\bar{b})$, and $c$ is an element of $\mathcal{M}$, then there is $d$ in $\mathcal{M}$
such that $tp^{\mathcal{M}}(\bar{a}c)=tp^{\mathcal{M}}(\bar{b}d)$.
\end{defi}

A slightly stronger notion is that of strong $\kappa$-homogeneity. 

\begin{defi} Let $\kappa$ be an infinite cardinal. An $\mathcal{L}$-structure $\mathcal{M}$ is strongly $\kappa$-homogeneous 
if whenever $\beta<\kappa$ and $\bar{a},\bar{b}$ are $\beta$-tuples from $\mathcal{M}$ with $tp^{\mathcal{M}}(\bar{a})=tp^{\mathcal{M}}(\bar{b})$ 
then there is an automorphism of the $\mathcal{L}$-structure $\mathcal{M}$ which sends $\bar{a}$ to $\bar{b}$. 
\end{defi}

In both cases an $\mathcal{L}$-structure $\mathcal{M}$ is called homogeneous (respectively strongly homogeneous) if it is $|\mathcal{M}|$-homogeneous (respectively strongly $|\mathcal{M}|$-homogeneous).
 
One can easily see that if a structure is strongly $\kappa$-homogeneous then it is $\kappa$-homogeneous. A less trivial 
observation is the following
\begin{prop}
A structure is homogeneous if and only if it is strongly homogeneous. 
\end{prop}
The proof uses a classical back and forth construction.     

There are a few results in model theory connecting homogeneity with other model theoretic properties. 
For example, a $\kappa$-saturated model is $\kappa$-homog\-e\-n\-e\-o\-us. 
Also, an atomic model (that is a model $\mathcal{A}$ for which $tp^{\mathcal{A}}(\bar{a})$ 
is isolated for all $\bar{a}\in\mathcal{A}^n$) is $\aleph_0$-homogeneous. 
Along the same line, one can see that all models of an $\aleph_0$-categorical 
theory are $\aleph_0$-homogeneous. 

The following example is more natural, and also points out the importance of the language considered:  
indeed, as shown in \cite{NadelStavi}, any model of $\mathcal{T}h(\mathbb{Z},1)$ is $\aleph_0$-homogeneous. 
But not all models of $\mathcal{T}h(\mathbb{Z})$ are $\aleph_0$-homogeneous, so it is essential that 
$1$ be named. More generally, if $G$ is an infinite abelian group of bounded exponent, then $\mathcal{T}h(G)$ is $\aleph_0$-categorical (see \cite[Corollary 4.4.4]{MarkerModelTheory}). Hence by the previous remarks, all models of $\mathcal{T}h(G)$ are $\aleph_0$-homogeneous. 

We finally give an example demonstrating that the notions of strong $\kappa$-homogen\-e\-i\-t\-y 
and $\kappa$-homogeneity do not coincide in general. 
\begin{ex}
Let $\mathcal{L}=\{E\}$ be the language consisting of 
a single binary relation, and let $\mathcal{M}$ be an $\mathcal{L}$-structure 
where $E$ is interpreted as an equivalence relation with two equivalence classes, one of 
cardinality $\aleph_0$ and the other of cardinality $\aleph_1$. Then  
$\mathcal{T}h(\mathcal{M})$ is $\aleph_0$-categorical, thus $\mathcal{M}$ is $\aleph_0$-homogeneous. But if we take $a,b\in\mathcal{M}$ with $\lnot E^{\mathcal{M}}(a,b)$, then 
$tp^{\mathcal{M}}(a)=tp^{\mathcal{M}}(b)$, and it is easy to check that there is no 
automorphism of $\mathcal{M}$ taking $a$ to $b$. 
\end{ex}

\subsection{Stable groups}
Stability theory is one of the main recent developments of model theory. It was introduced by Shelah for the purpose of distinguishing wild structures from tame ones from the point of view of first-order logic. The stability of a first-order theory can be given a number of equivalent definitions: we give one which states that it is impossible to "encode" an infinite linear order in a model of the theory.
\begin{defi} A complete theory $T$ is stable if there do not exist a formula $\delta(x, y)$, a model $\mathcal{M}$ of $T$, and sequences $(a_i)_{i <\omega}$ and $(b_i)_{i \in \omega}$ of elements of $\mathcal{M}$ such that $\mathcal{M} \models \delta(a_i, b_j)$ if and only if $i < j$. 
\end{defi}

Structures whose first-order theory is stable are in some sense well-behaved, whereas it is often very difficult to say anything about unstable theories. For an introduction to stability theory, the reader is referred to \cite{PillayStability}.

Within this framework, stable groups hold a special position. There is an independent theory 
developed for them including concepts such as generic types and connected components (see \cite{PoizatStableGroups} or turn to the introduction of \cite{PillayForking}). For the benefit of the reader we will now explain a few model theoretic notions about groups, with some emphasis on stable groups.

In model theory, a group $(G,\cdot)$ is a structure 
equipped with a group operation, but possibly also with some additional
relations and functions. Even when we do not have explicitly any additional
relations or functions, all the sets $X \subseteq G^n$ which are definable over the group under
consideration will be part of our structure. 

We define a stable group to be a group definable in a stable theory. By this we
mean that $(G,\cdot)$ is definable in a model $\mathcal{M}$ of the stable theory $T$, and it may be
equipped with some or all of the structure induced from $\mathcal{M}$. A typical example is 
when $G$ is an algebraic group over an algebraically closed field, and we think 
of $G$ as equipped with predicates for all Zariski closed subsets of $G^n$. 

The simplest case is of course when the group coincides with the ambient structure, and indeed this is the case for non abelian free groups.

\begin{defi}
Let $G$ be a 
group. Let $X$ be a definable subset of $G$.
We say that $X$ is left-generic (right-generic) if finitely many left 
(right) translates of $X$ by elements of $G$ cover $G$.
\end{defi}

Note that for a definable subset $X$ of a stable group $G$, it can be shown that 
$X$ is left generic iff $X$ is right generic, so for stable groups we will simply say generic. 
We will also say that a formula $\phi(x)$ with one free variable is generic in $G$ if the set 
$X =\{ g \in G \mid G \models \phi(g)\}$ it defines is generic. 

\begin{defi}
Let $G$ be a group. We say that $G$ is connected if it has no definable
proper subgroup of finite index.
\end{defi}

\begin{defi}
Let $G$ be a stable group. Let $g$ be an element of $G$, and let $A$ be a set of parameters
from $G$. We say that $tp^G(g/A)$ is a generic type if every formula in $tp^G(g/A)$ is generic.
\end{defi}

We have the following useful fact:
\begin{fact} \label{ConnectedIffUniqueGenericType}
Let $G$ be a stable group. Then $G$ is connected if and only if there is over any set of parameters a unique generic type.
\end{fact}

\subsection{Free groups}

Let $\kappa, \lambda$ be cardinals, and let $\F_{\kappa}$ denote the free group of rank $\kappa$. If $\kappa \leq \lambda$, any injection of $\kappa$ into $\lambda$ induces an embedding of $\F_{\kappa}$ into $\F_{\lambda}$ as a free factor. 

We say that an element of $\F_{\kappa}$ is primitive if it is part of some basis of $\F_{\kappa}$, analogously, a set of elements of $\F_{\kappa}$ is primitive if it can be extended to some basis of $\F_{\kappa}$. As mentioned above, Sela proved in \cite{Sel6} that finitely generated free groups of rank at least $2$ have the same first-order theory. In fact he showed the following 
\begin{thm}\label{FmElementaryInFn} Let $m, n \in \N$ with $2\leq m \leq n$. Any embedding of $\F_m$ into $\F_n$ as a free factor is elementary.
\end{thm}

A sequence of $\mathcal{L}$-structures $(\mathcal{M}_i)_{i \in I}$ is an elementary chain if for any $i, j$ with $i \leq j$, the structure $\mathcal{M}_i$ is an elementary substructure of $\mathcal{M}_j$. It is easy to see that the union $\mathcal{M}_{I} = \bigcup_{i \in I} \mathcal{M}_i$ can be seen as an $\mathcal{L}$-structure, and it is a classical result of model theory that $\mathcal{M}_i$ is an elementary substructure of $\mathcal{M}_{I}$ for each $i \in I$ (see \cite{MarkerModelTheory}). 

Theorem \ref{FmElementaryInFn} shows that $(\F_n)_{n \in \omega}$ with the canonical embedding of $\F_m$ in $\F_n$ as a free factor is an elementary chain. Now clearly the union of $(\F_n)_{n \in \omega}$ is isomorphic to $\F_{\aleph_0}$: we see that for $n \geq 2$, any embedding of $\F_n$ as a free factor in $\F_{\aleph_0}$ is elementary. In particular, $\F_{\aleph_0}$ is elementary equivalent to $\F_n$.

It is straightforward to see that Theorem \ref{FmElementaryInFn} also holds for free groups of infinite rank, using for example the Tarski-Vaught test (see \cite{MarkerModelTheory}):
\begin{thm}[Tarski-Vaught Test]
Suppose that $\mathcal{M}$ is a substructure of $\mathcal{N}$. Then $\mathcal{M}$ 
is an elementary substructure of $\mathcal{N}$ iff for any formula $\phi(x,\bar{y})$ and 
$\bar{a}\in\mathcal{M}$, if there is $b\in \mathcal{N}$ 
such that $\mathcal{N}\models\phi(b,\bar{a})$ then there is $c\in\mathcal{M}$ 
such that $\mathcal{N}\models\phi(c,\bar{a})$.
\end{thm}

\begin{lemma}\label{FKappaElementaryInFLambda}
Let $\kappa,\lambda$ be infinite cardinals, with $\kappa\leq \lambda$. Then any embedding of $\F_{\kappa}$ in $\F_{\lambda}$ as a free factor is elementary.
\end{lemma}
\begin{proof}
Let $\phi(x,\bar{y})$ be an $\mathcal{L}$-formula, let $\bar{a}\in \F_{\kappa}$, 
and suppose there exists $b \in \F_{\lambda}$ such that $\F_{\lambda}\models\phi(b,\bar{a})$.  
By the Tarski-Vaught test, we only need to show that there is $c\in \F_{\kappa}$ 
such that $\F_{\lambda}\models\phi(c,\bar{a})$. This will be an immediate consequence of the following easy claim.
\begin{claim} For any finite primitive set $P$ of $\F_{\kappa}$ and any element $b$ of $\F_{\lambda}$, 
there is an automorphism $f$ of $\F_{\lambda}$ fixing $P$ pointwise such that $f(b)\in \F_{\kappa}$.
\end{claim}
Now, the elements of the tuple $\bar{a}$ are words over some finite primitive set $P$ of $\F_{\kappa}$ since 
by hypothesis $\bar{a}\in \F_{\kappa}$. By the claim above, there is an automorphism of $\F_{\lambda}$ which fixes $P$ pointwise and such that $f(b)\in \F_{\kappa}$. 
Thus $\F_{\lambda}\models \phi(b,\bar{a})$ iff $\F_{\lambda}\models \phi(f(b),\bar{a})$, and this finishes the proof. 
\end{proof}

From now on we denote by $T_{fg}$ the common theory of the non abelian free groups. The following 
astonishing result has been proved by Sela in \cite{SelaStability}.

\begin{thm}
$T_{fg}$ is stable.
\end{thm}
      
The above theorem clearly introduces new tools in the studying of free groups. In \cite{PillayForking}, Pillay observed that the free group $\F_{\aleph_0}$ is connected. As connectedness is a first order property, this shows that any model 
of $T_{fg}$ is a connected stable group. 

By Fact \ref{ConnectedIffUniqueGenericType}, connectedness of a stable group implies that there is 
a unique generic type over any set of parameters. We denote by $p_0$ the 
unique generic type of $T_{fg}$ over $\emptyset$.  

Towards the understanding of $p_0$, Pillay proved in \cite{PillayGenericity} the following:
\begin{thm}
Let $a$ be an element of $\F_n$ with $n\geq 2$. If $a$ realizes $p_0$ in $\F_n$, then $a$ is primitive. 
\end{thm}

So the primitives of $\F_n$ are exactly the elements which realize $p_0$ in $\F_n$. The proof uses essentially 
the characterization given in \cite{PerinElementary} of elementary embeddings in finitely generated free groups.  
Some more results concerning $p_0$ are proved in \cite{PillayForking, PillayGenericity, SklinosGenericType}.

\section{Modular group and JSJ decomposition} \label{ModularGroupAndJSJSec}
Let $G$ be a finitely generated group. Let $\Lambda$ be a one-edge cyclic splitting of $G$, that is, a decomposition of $G$ as an amalgamated product $A*_C B$ or as an HNN extension $A *_C$ over an infinite cyclic subgroup $C$ of $G$. Recall that a subgroup $H$ of $G$ is said to be elliptic in $\Lambda$ if it is contained in a conjugate of $A$ or $B$ (respectively in a conjugate of $A$ for the case of an HNN extension).

\begin{defi} The Dehn twist of $\Lambda$ by an element $\gamma$ of the center of $C$ is the automorphism which restricts to the identity on $A$ and to conjugation by $\gamma$ on $B$ (respectively the automorphism of $G$ which restricts to the identity on $A$ and sends the stable letter $t$ to $t \gamma$). 
\end{defi}

In this section we will often assume the group $G$ to be torsion-free hyperbolic, though the definitions and results we give extend to more general cases. However, we will only use them in this specific setting which simplifies both the definitions and the results.   

\begin{defi} Let $G$ be a freely indecomposable torsion-free hyperbolic group.
The (cyclic) modular group $\Mod(G)$ of $G$ is the subgroup of $\Aut(G)$ generated by Dehn twists of one-edge cyclic splittings of $G$. 

Let $G$ be a torsion-free hyperbolic group which is freely indecomposable with respect to some subgroup $H$.
The (cyclic) modular group $\Mod_H(G)$ of $G$ relative to $H$ is the subgroup of $\Aut_H(G)$ generated by the Dehn twists fixing $H$ of the one-edge cyclic splittings of $G$ in which the subgroup $H$ is elliptic.
\end{defi}

It is a results of Rips and Sela (see \cite{RipsSelaHypI}) that in the torsion-free hyperbolic case, the modular group has finite index in the group of automorphisms.

The idea of the JSJ decomposition is to encode all possible splittings of $G$ over a given class ${\cal A}$ of subgroups in a single graph of group $\Gamma$ (for definitions and results about graphs of groups see \cite{SerreTrees}). Various results proving the existence of a JSJ decomposition and describing its properties for different hypotheses on $G$ and ${\cal A}$ are obtained in \cite{RipsSelaJSJ, FujiwaraPapasoglu, DunwoodySageev, Bowditch}. We use the unifying framework of Guirardel and Levitt, developed in \cite{GuirardelLevittJSJI, GuirardelLevittJSJII}. We now summarize briefly the definitions and results we will use. 

Given a class ${\cal A}$ of subgroups of $G$ which is stable under conjugation and taking subgroups, we consider the class of all ${\cal A}$-trees, namely all the simplicial trees endowed with an action of $G$ whose edge stabilizers are in the class ${\cal A}$. One might also be interested in the JSJ relative to some subgroup $H$ of $G$: in this case, one considers only the ${\cal A}$-trees in which $H$ is elliptic (that is, fixes a vertex). We will call such trees $({\cal A}, H)$-trees.

Given ${\cal A}$-trees $T$ and $T'$, we say that $T$ \textbf{dominates} $T'$ if there exists a $G$-equivariant continuous map $T \to T'$, and that $T$ \textbf{refines} $T'$ (or \textbf{collapses to} $T'$) if this map consists in collapsing some of the edges of $T$ to vertices. The \textbf{deformation space} of an ${\cal A}$-tree $T$ is the set of all ${\cal A}$-trees $T'$ such that $T$ dominates $T'$ and $T'$ dominates $T$. An ${\cal A}$-tree is \textbf{universally elliptic} if its edge stabilizers are elliptic in every ${\cal A}$-tree. If $T$ is a universally elliptic ${\cal A}$-tree, and $T'$ is any ${\cal A}$-tree, it is easy to see that there is a tree $\hat{T}$ which refines $T$ and dominates $T'$ (see \cite[Lemma 3.2]{GuirardelLevittJSJI}).

A \textbf{JSJ tree} is a universally elliptic ${\cal A}$-tree which dominates any other universally elliptic tree. All JSJ-trees belong to a same deformation space, that we denote ${\cal D}_{JSJ}$. Guirardel and Levitt show that if $G$ is finitely presented, the JSJ deformation space always exists, without any restrictions on ${\cal A}$ (see \cite[Theorem 4.2]{GuirardelLevittJSJI}). This extends to the relative case (replacing 'every ${\cal A}$-tree' by  'every $({\cal A}, H)$-tree' where appropriate in the definitions), provided $H$ is also finitely generated.

The real heart of the theory of JSJ decompositions is to describe the properties of the JSJ trees, and sometimes to find a canonical tree in ${\cal D}_{JSJ}$. A vertex stabilizer in a (relative) JSJ tree is said to be \textbf{rigid} if it is elliptic in any ${\cal A}$-tree (respectively any $({\cal A}, H)$-tree), and \textbf{flexible} if not. Interpreting results of \cite{RipsSelaJSJ, DunwoodySageev, FujiwaraPapasoglu} in this framework gives a description of flexible vertices under some conditions on ${\cal A}$ (see \cite[Theorem 7.7]{GuirardelLevittJSJI} and \cite[Theorem 7.38]{GuirardelLevittJSJI} for the relative case). We will give this description in the special case where $G$ is torsion-free hyperbolic and freely indecomposable (respectively freely indecomposable with respect to the finitely generated subgroup $H$). We first give

\begin{defi} \label{SurfaceTypeVertex} Fix a group $G$. Let $T$ be a tree endowed with an action of $G$ (respectively an action of $G$ in which $H$ is elliptic). We say that a vertex $v$ of $T$ is of surface type if 
\begin{itemize}
    \item the stabilizer of $v$ is the fundamental group $S$ of a hyperbolic surface with boundary $\Sigma$;
    \item each incident edge stabilizer is contained in a boundary subgroup of $S$;
    \item each maximal boundary subgroup of $S$ contains as a subgroup of finite index the stabilizer of an incident edge (or a conjugate of $H$ in the relative case).
\end{itemize}
If $\Lambda$ is the graph of group corresponding to $T$, we also say that the  vertex of $\Lambda$ corresponding to $v$ is of surface type.
\end{defi}

We now have:
\begin{thm} \label{JSJ} Let $G$ be torsion-free hyperbolic and freely indecomposable (respectively freely indecomposable with respect to a non trivial finitely generated subgroup $H$). Let $T$ be a (relative) JSJ tree over the class ${\cal A}$ of cyclic subgroups. Then the flexible vertices of $T$ are of surface type. 
\end{thm}

The first property of JSJ trees we will use is that its vertex groups are preserved in some sense by modular automorphisms.
\begin{lemma} \label{ModAndJSJ} Let $G$ be torsion-free hyperbolic and freely indecomposable (respectively freely indecomposable with respect to a non trivial finitely generated subgroup $H$). Let $T$ be a (relative) JSJ tree over the class ${\cal A}$ of cyclic subgroups. An element of $\Mod(G)$ (respectively $\Mod_H(G)$) restricts to conjugation on each rigid vertex stabilizer of $T$, and sends flexible vertex stabilizers isomorphically on conjugates of themselves.
\end{lemma}

\begin{proof} Let $T_0$ be the ${\cal A}$-tree corresponding to a one-edge cyclic splitting of $G$ along a cyclic subgroup $C$, and let $\tau$ be the Dehn twist corresponding to an element $\gamma$ of $C$. Let $v$ be a vertex of $T$ with vertex group $V$. If $v$ is elliptic in $T_0$ (in particular if $v$ is rigid), $\tau|_V$ is just a conjugation. 

Suppose now that $V$ is not elliptic in $T_0$: in particular, $v$ must be flexible, so it is of surface type. We will show that (some conjugate of) $\gamma$ represents a simple closed curve on the corresponding surface $\Sigma$. The boundary subgroups of $V$ are elliptic in any ${\cal A}$-tree since $T$ is universally elliptic. Thus by Proposition III.2.6 of \cite{MorganShalen}, $V$ inherits from its action on $T_0$ a splitting which is dual to a set of essential simple closed curves on the surface corresponding to $V$. Thus, up to conjugation, $\gamma$ is the root of an element $c$ of $V$ which represents an essential simple closed curve on $\Sigma$. Now the orbit of the vertex $v$ under $\langle \gamma \rangle$ is finite, so $\gamma$ must fix a vertex $w$ of $T$. If $v \neq w$, the path between $v$ and $w$ is fixed by $c$, so $c$ is a boundary element of $V$. But this contradicts essentiality of the simple closed curve it represents. Hence $\gamma$ fixes $v$, so it lies in $V$ and represents a simple closed curve on the corresponding surface. The Dehn twist $\tau$ by $\gamma$ is thus an automorphism of $V$, so $\tau$ sends $V$ isomorphically to a conjugate of itself.  
\end{proof}

We will pick among the trees in ${\cal D}_{JSJ}$ a particular JSJ tree with specific properties: the tree of cylinders. It turns out to be precisely the JSJ decomposition given by Bowditch in \cite{Bowditch}. In our setting, the cylinders of an ${\cal A}$-tree $T$ (respectively of an $({\cal A}, H)$-tree) are the equivalence classes of edges under the relation given by commensurability of stabilizers. Cylinders are subtrees of $T$, and the tree of cylinders $T_C$ of $T$ (which can be shown to depend only on the deformation space of $T$) is built as follows: its set of vertices is the union of the set $C(T)$ of cylinders of $T$ and the set $\partial C(T)$ of vertices of $T$ which lie in two distinct cylinders, and there is an edge between a vertex $w$ of $C(T)$ and a vertex $v$ of $\partial C(T)$ if $v$ lies in the cylinder corresponding to $w$. The following properties of $T_C$ are given in \cite[Theorem 2]{GuirardelLevittTreeOfCylinders}.

\begin{thm} \label{CylinderJSJ} Let $G$ be torsion-free hyperbolic and freely indecomposable (respectively freely indecomposable with respect to a non trivial finitely generated subgroup $H$). Let ${\cal A}$ be the class of cyclic subgroups of $G$. Let $T_C$ be the tree of cylinders of an ${\cal A}$-tree $T$. Then $T_C$ is itself an ${\cal A}$-tree which belongs to the same deformation space as $T$, in particular if $T$ is a JSJ tree, so is $T_C$. Moreover, $T_C$ is $2$-acylindrical.
\end{thm}

From now on, we will call any decomposition of a group $G$ corresponding to a (relative) JSJ tree \textbf{a} (relative) JSJ decomposition of $G$, but \textbf{the} (relative) JSJ decomposition of a group $G$ will refer to the decomposition corresponding to the tree of cylinders of the (relative) JSJ deformation space.

\begin{rmk} It is straightforward to see by construction of the tree of cylinders that in the tree of cylinders of the JSJ deformation space, all the stabilizers of edges incident to a flexible vertex have index $1$ in the corresponding boundary subgroup, and two incident edges whose stabilizers correspond to conjugate boundary subgroups are in the same orbit. 
\end{rmk}

Finally, we will need the following result.
\begin{prop} \label{ModGroupStabilizes} Let $G$ be a torsion-free hyperbolic group. Let $H$ be a subgroup of $G$ with respect to which $G$ is freely indecomposable. There exists a finitely generated subgroup $H_0$ of $H$ such that $G$ is freely indecomposable with respect to $H_0$, and $H$ is elliptic in the JSJ decomposition of $G$ relative to $H_0$. Moreover, we have $\Mod_{H_0}(G) = \Mod_H(G)$.
\end{prop}

\begin{proof} Consider an exhausting chain of finitely generated subgroups $H_n$ of $H$. By Lemma 4.23 in \cite{PerinElementary} (see also Lemma 6.6 of \cite{GuirardelLevittJSJII}), we may assume that $G$ is freely indecomposable with respect to all the subgroups $H_n$.

%

For each $n$, we let $T_n$ be a tree in the JSJ deformation space of $G$ relative to $H_n$. The tree $T_n$ does not necessarily dominate $T_m$ for $m \geq n$, but we claim that there exists for each $n$ a refinement $\hat{T}_n$ of $T_n$ which dominates $\hat{T}_m$ for all $m \geq n$. 

Suppose we do have such a sequence $\hat{T}_n$. Let $S_n$ be the tree of cylinders of $\hat{T}_n$: it is in the same deformation space as $\hat{T}_n$. By functoriality of the trees of cylinders (see \cite[Proposition 4.11]{GuirardelLevittTreeOfCylinders}), there is a $G$-equivariant cell map $f_n: S_n \to S_{n+1}$. This shows that the number of orbits of edges of $S_n$ is non increasing, thus we may assume it is constant. If $f_n$ is not an isomorphism, there must be some foldings, hence some edges whose stabilizer is not maximal cyclic in the adjacent vertex stabilizers. Each folding increases strictly the index of an edge stabilizer in its maximal cyclic subgroup (recall we are in a torsion-free hyperbolic group): this can only happen finitely many times. The sequence $S_n$ eventually stabilizes, so there exists $n_0$ such that $H_n$ is elliptic in $S_{n_{_0}}$, and thus in $T_{n_{_0}}$, for all $n$ large enough.

Finally, the vertex group containing $H_{n_{_0}}$ of the JSJ tree $T_{n_{_0}}$ is elliptic in any one edge cyclic splitting of $G$ in which $H_{n_{_0}}$ is elliptic. This means that $H$ is elliptic in any such splitting, which implies that $\Mod_{H_{n_{_0}}}(G) \subseteq \Mod_H(G)$. The other inclusion is immediate.  

There remains only to show that we can build the sequence $\hat{T}_n$. The idea is the following: we start by refining $T_1(1) = T_1$ to a tree $T_1(2)$ which dominates $T_2$. Then, we refine $T_2$ to a tree $T_2(3)$ which dominates $T_3$. We now want to see that we can refine $T_1(2)$ to a tree $T_1(3)$ which dominates also $T_2(3)$. 

%
\begin{center}
\xymatrix{
&\vdots&\vdots                    & \vdots                  & \vdots                     &  \\
&T_k(k)&T_k(k+1)\ar@{->>}[l]\ar[d]& \ldots\ar@{->>}[l]\ar[d]&T_k(n)\ar@{->>}[l]\ar[d]    & \ar@{-->>}[l] \hspace{0.2cm}{\mathbf ?}\hspace{0.2cm} \ar@{-->}[d]  \\
&      &T_{k+1}(k+1)              & \ldots\ar@{->>}[l]\ar[d]&T_{k+1}(n)\ar@{->>}[l]\ar[d]&T_{k+1}(n+1) \ar@{->>}[l]\ar[d]\\
&      &                          &\vdots     \ar[d]        &\vdots      \ar[d]          &\vdots \ar[d]    \\
&      &                          &T_{n-1}(n-1)             &T_{n-1}(n)\ar@{->>}[l]\ar[d]&T_{n-1}(n+1) \ar@{->>}[l]\ar[d] \\
&      &                          &                         &T_{n}(n)                    &T_{n}(n+1)\ar@{->>}[l]\ar[d] \\
&      &                          &                         &                            &T_{n+1}(n+1)  \\
}
\end{center}

More formally, we want to build $({\cal A}, H_k)$-trees $T_k(n)$ for all $n \geq k$ such that 
\begin{enumerate}
\item \label{UnivEllCon} $T_k(n)$ is $({\cal A}, H_k)$-universally elliptic;
\item $T_k(n)$ dominates $T_{k+1}(n)$;
\item $T_k(n)$ is a refinement of $T_k(n-1)$;
\end{enumerate}
(see the commutative diagram, where double arrows represent collapse).

We proceed by (upwards) induction on $n$. We let $T_1(1) = 1$. Suppose we have built the trees $T_k(j)$ for $k \leq j \leq n$. We now want to build the trees $T_k(n+1)$ for $k \leq n+1$. We start by letting $T_{n+1}(n+1) = T_{n+1}$: it satisfies \ref{UnivEllCon}. Now, we proceed by (downwards) induction on $k$: suppose we have built the trees $T_j(n+1)$ for $n +1 \geq j \geq k+1$. The tree $T_k(n)$ is $({\cal A}, H_k)$-universally elliptic by induction hypothesis, and $T_{k+1}(n+1)$ is a $({\cal A}, H_k)$-tree, so we can refine $T_k(n)$ to a tree $T_k(n+1)$ dominating $T_{k+1}(n+1)$. The edge groups of $T_k(n+1)$ which are not already edge groups of $T_k(n)$ are contained in edge groups of $T_{k+1}(n+1)$, which by induction on $k$ are elliptic in any $({\cal A}, H_{k+1})$-tree. Thus $T_k(n+1)$ satisfies condition \ref{UnivEllCon}.

For each $k$, the sequence of refinements $(T_k(j))_{j \in \N}$ eventually stabilizes (by Dunwoody accessibility, see for example Proposition 4.3 of \cite{GuirardelLevittJSJI}). Thus we get a sequence $\hat{T}_n$ of refinements of the JSJ trees $T_n$ for which $\hat{T}_n$ dominates $\hat{T}_{n+1}$.  
\end{proof}

\section{Factor sets and test elements} \label{FactorSetSec}
The following results, whose proofs are all based on the shortening argument, describe various properties of homomorphisms $G \to \Gamma$ between torsion-free hyperbolic groups, and will therefore be crucial in understanding homogeneity properties of such groups. They are all variations on results proved in \cite{Sel1} for the free case and in \cite{Sel7} for the general torsion-free hyperbolic case. Note that most of these results are valid if we only assume $G$ to be finitely generated, but this would require a more general definition of the modular group, and some subtle arguments are required to deal with the axial components in the shortening. 


The first result deals with embeddings $G \hookrightarrow \Gamma$.
\begin{thm} \label{FiniteInjections} Let $G$ and $\Gamma$ be torsion-free hyperbolic groups. Let $H$ be a subgroup of $G$ with respect to which $G$ is freely indecomposable, and let $f:H \to \Gamma$ be an embedding of $H$ into $\Gamma$.

Then there exists a finite set $\{i_j: G \to \Gamma\}_{1 \leq j \leq s}$ of embeddings of $G$ into $\Gamma$ such that for any embedding $\theta: G \to \Gamma$ restricting to $f$ on $H$, there exists an element $\sigma$ of $\Mod_{H}(G)$ and an element $\gamma$ in the centralizer $Z(f(H))$ of $f(H)$ such that $\theta = \Conj(\gamma) \circ i_j \circ \sigma$ for some $j$. 
\end{thm}

For monomorphisms, we get
\begin{thm} \label{RelativeCoHopf} Let $\Gamma$ be a torsion-free hyperbolic group which is freely indecomposable with respect to a non trivial subgroup $H$. Then any injective homomorphism $\Gamma \to \Gamma$ which restricts to the identity on $H$ is an isomorphism.
\end{thm}

The result also holds when $H$ is trivial, provided $\Gamma$ is not $2$-ended (see Theorem 4.4 of \cite{SelaHypII}) but the proof is much easier in the relative case.

\begin{rmk} If $H$ is finitely generated, Theorem \ref{RelativeCoHopf} says exactly that a finite generating set for $H$ is a test tuple for monomorphisms (see Definition 2.3 of \cite{NiesF2}). To show homogeneity of $\F_2$, Nies uses the fact that for any tuple $\bar{g}$ in $\F_2$, either $\F_2$ is freely indecomposable with respect to the subgroup $\langle \bar{g} \rangle$ (so $\bar{g}$ is a test tuple), or the elements of $\bar{g}$ are all powers of a same primitive element. The second crucial factor in Nies' proof is that the injectivity of a morphism $\F_2 \to \F_2$ can be expressed in first-order (it is enough to say that the images of the generators do not commute). Both these properties do not hold in higher rank.
\end{rmk}

Finally, non injective morphisms are dealt with by 
\begin{thm} \label{FactorSet} Let $G$ and $\Gamma$ be torsion-free hyperbolic groups. Let $H$ be a subgroup of $G$ with respect to which $G$ is freely indecomposable, and let $f:H \to \Gamma$ be a homomorphism.

Then there exists a finite set $\{\eta_i: G \to M_i\}_{1 \leq i \leq s}$ of proper quotients of $G$ such that for any non injective homomorphism $\theta: G \to \Gamma$ which restricts to $f$ on $H$, there exists an element $\sigma$ of $\Mod_{H}(G)$ such that $\theta \circ \sigma$ factors through one of the maps $\eta_i$. 
\end{thm}

Before proving these three results, we will deduce from Theorem \ref{FiniteInjections} the following corollary:
\begin{cor} \label{SetOfInjStabilizes} Let $G$ and $\Gamma$ be torsion-free hyperbolic groups. Let $H$ be a subgroup of $G$ with respect to which $G$ is freely indecomposable, and let $f$ be an embedding of $H$ into $\Gamma$. There is a finitely generated subgroup $H_0$ of $H$ such that any embedding $\theta: G \to \Gamma$ which restricts to $f$ on $H_0$ restricts to $f$ on $H$.
\end{cor}

\begin{proof}[Proof of corollary] We may assume $H$ infinitely generated, hence non abelian. We write $H$ as an increasing union of finitely generated subgroups $H_n$. For $n$ large enough, by Lemma \ref{ModGroupStabilizes} we may assume that $G$ is freely indecomposable with respect to $H_n$, and that $\Mod_{H_n}(G) = \Mod_H(G)$.
Since $G$ is freely indecomposable with respect to $H_n$, Theorem \ref{FiniteInjections} gives a finite set $I_{H_n}$ of injective morphisms $G \to \Gamma$ which restrict to the identity on $H_n$. For $n$ large enough $H_n$ is non abelian, so that $Z(f(H_n))$ is trivial. Any injective morphism $\theta: G \to \Gamma$ which restricts to $f$ on $H_n$ is equal to one of the elements of $I_{H_n}$ after precomposition by an element of $\Mod_{H_n}(G)$. In particular, this applies to the injective morphisms contained in $I_{H_{n+1}}$, so we may assume $I_{H_{n+1}} \subseteq I_{H_n}$ since $\Mod_{H_n}(G) = \Mod_{H_{n+1}}(G)$.

This non increasing chain of finite sets stabilizes so we may assume it is constant. Suppose $\theta$ restricts to $f$ on $H_n$. Then it is the composition of an element of $\Mod_{H_n}(G)$ with an element $i$ of $I_{H_n}$. But any element of $\Mod_{H_n}(G)$ restricts to the identity on $H$, and $i$ lies in $I_{H_n}$ for all $n$, so it also restricts to the identity on $H$. This proves the result. 
\end{proof}

We will now give outlines of the proofs of Theorems \ref{RelativeCoHopf}, \ref{FactorSet} and \ref{FiniteInjections}.
Let  $G$ and $\Gamma$ be torsion-free hyperbolic groups. Let $H$ be a subgroup of $G$ with respect to which $G$ is freely indecomposable, and fix a morphism $f:H \to \Gamma$. We fix finite generating sets $S$ and $\Sigma$ for $G$ and $\Gamma$ respectively, and we denote by $|.|_{\Sigma}$ the word length with respect to $\Sigma$. 

\begin{defi} We say that a homomorphism $\theta: G \to \Gamma$ which restricts to $f$ on $H$ is short with respect to $H$ if for any element $\tau$ of $\Mod_H(G)$ and any element $\gamma$ of the centralizer $Z(f(H))$ of $f(H)$ in $\Gamma$, we have:
$$ \max_{g \in S} |\theta(g)|_{\Sigma} \leq \max_{g \in S} |\gamma \theta(\tau(g)) \gamma^{-1}|_{\Sigma}.$$
\end{defi}

\begin{defi} A sequence $(\theta_n)_{n \in \N}$ of group homomorphisms $\theta_n: G_1 \to G_2$ is said to be stable if for any element $g$ of $G_1$, either $g$ lies in $\Ker \; \theta_n$ for all $n$ large enough, or $g$ lies outside of $\Ker\; \theta_n$ for all $n$ large enough. The set of elements for which the first alternative holds is called the stable kernel of the sequence $\theta_n$, and is denoted by $\underleftarrow{\Ker}\; \theta_n$.
\end{defi}

\begin{rmk} If $G_1$ is countable, any sequence of morphisms $G_1 \to G_2$ has a stable subsequence.
\end{rmk}

\begin{prop} \label{ShorteningQuotientsAreProper} Let $G$ and $\Gamma$ be torsion-free hyperbolic groups, assume $G$ is not cyclic. Let $H$ be a subgroup of $G$ with respect to which $G$ is freely indecomposable, and let $f$ be a morphism $H \to \Gamma$.
Suppose $(\theta_n)_{n \in \N}$ is a stable sequence of distinct morphisms $G \to \Gamma$ which restrict to $f$ on $H$, and are short with respect to $H$. Then its stable kernel $\underleftarrow{\Ker}\; \theta_n$ is non trivial.
\end{prop}

The proof of this is an instance of Rips and Sela's classical shortening argument developed in \cite{RipsSelaHypI} (see \cite{WiltonThesis} or \cite{ThesisPerin} for a more detailed exposition), thus we only outline it and detail the arguments which are not classical. 

\begin{proof}[Outline of the proof] Let us assume by contradiction that the stable kernel of the sequence $(\theta_n)_{n \in \N}$ is trivial (this implies in particular that $f$ is injective). Denote by $X$ the Cayley graph of $\Gamma$ with respect to $\Sigma$. The group $G$ acts on $X$ via $\theta_n$. The displacement function $\Delta_n: X \to \N$ associated to this action is given by 
$$\Delta_n(x)=\max_{g \in S} d_X( x, \theta_n(g)\cdot x) = \max_{g \in S} |x^{-1}\theta_n(g)x|_{\Sigma}.$$ 

Let $\mu_n$ denote the minimum reached by this function on $Z(f(H))$ (where $Z(f(H))$ is seen as a set of vertices in $X$). Note that the shortness of the maps $\theta_n$ implies that this minimum is reached for $x=1_n$ the vertex of $X$ representing the identity element. We thus have $\mu_n = \max_{g \in S} |\theta_n(g)|_{\Sigma}$. The maps $\theta_n$ are pairwise distinct so the sequence $\mu_n$ tends to infinity.

We denote by $X[\theta_n]$ the space $X$ endowed with the metric $d_n = d_X/\mu_n$, with the action of $G$ via $\theta_n$, and with basepoint the vertex $1_n$ corresponding to the identity element. 

After extraction, the sequence $X[\theta_n]$ converges to a pointed real $G$-tree $(T,x)$, that we can assume minimal. The triviality of the stable kernel of $(\theta_n)_{n \in \N}$ implies in particular that the image $\theta_n(G)$ is eventually non cyclic, and in this case it can be shown that the action of $G$ on $T$ is faithful, that tripod stabilizers are trivial and that arc stabilizers are abelian. These last two properties imply that the action is superstable.

Let us show that it does not have a global fixed point. By choice of the scaling constant, the basepoint $x$ is not a global fixed point. On the other hand, since the maps $\theta_n$ are constant on $H$ and $\mu_n$ tends to infinity, $H$ fixes the basepoint. If there is a global fixed point $y$, then $H$ fixes the arc between $x$ and $y$, so in particular it is abelian. Since it is a subgroup of the torsion-free hyperbolic group $G$, it must in fact be cyclic or trivial. Let $g$ be an element of the generating set $S$ of $G$ such that $d(x, g \cdot x)$ is non zero, and pick $\epsilon$ much smaller than $d(x, g \cdot x)$. Let now $y_n$ approximate $y$ in $X[\theta_n]$ in an $\epsilon$-approximation with respect to $x$, $y$, $S$ and a generator for $H$. Now for $n$ large enough, we have $d_n(y_n, \theta_n(g)y_n)< d_n(1_n, \theta_n(g) 1_n)$, so that $\Delta_n(y_n) < \Delta_n(1_n)$.

If $H$ is trivial, $y_n$ lies in $Z(f(H))$ and this contradicts the fact that $1_n$ minimizes $\Delta$ on $Z(f(H))$. If $H$ is cyclic and generated by $h$, then for $n$ large enough $1_n$ and $y_n$ are very close to the axis of $h$ in $X[\theta_n]$. Thus, there exists an integer $K_n$ such that $z_n=f(h)^{K_n}$ is very close to $y_n$. Now the orbit of $y_n$ by $G$ has small diameter, thus so does the orbit of $z_n$: this implies that $\Delta_n(z_n) < \Delta_n(1_n)$, which contradicts shortness of $\theta_n$ since $z_n=f(h)^{K_n}$ is in $Z(f(H))$. We showed that the action does not admit a global fixed point. 

The action of $G$ on $T$ satisfies the hypotheses necessary to be analyzed by Rips theory (see Theorem 5.1 of \cite{GuirardelRTrees}), this gives us in particular a decomposition $\Lambda$ of $G$ as a graph of groups with surfaces in which the subgroup $H$ is elliptic.
We can now use classical shortening arguments (originally presented in \cite{RipsSelaHypI}, see also Theorem 4.28 of \cite{ThesisPerin}) to show that we can precompose $\theta_n$ by an element $\sigma$ of the modular group $\Mod(\Lambda)$ in such a way that $\theta_n \circ \sigma$ is shorter than $\theta_n$. This is a contradiction.
\end{proof}


\begin{proof}[Proof of Theorem \ref{FiniteInjections}] By Theorem \ref{ShorteningQuotientsAreProper}, any stable sequence of distinct short morphisms $\theta_n: G \to \Gamma$ which coincide with $f$ on $H$ has a non-trivial stable kernel. In particular, this means there are only finitely many short injective morphisms $G \to \Gamma$ restricting to $f$ on $H$. This proves the result.
\end{proof}

\begin{proof}[Proof of Theorem \ref{RelativeCoHopf}] Suppose $\theta$ is an injective yet non surjective morphism $\Gamma \to \Gamma$ which restricts to the identity on $H$. 

For each integer $n$, we consider the $n$-th power of $\theta$: it is an injective morphism $\theta^n: G \to G$. We pick an element $\tau_n$ of $\Mod_H(G)$, and $\gamma_n$ of $Z(H)$ such that the morphism $\hat{\theta}_n = \Conj(\gamma_n) \circ \theta^n \circ \tau_n$ is short with respect to $H$. Note that since $H$ is non trivial and $\Gamma$ is torsion-free hyperbolic, its centralizer is either trivial, or an infinite cyclic subgroup which contains $H$ as a finite index subgroup. Thus, up to replacing $\theta$ by a conjugate, and extracting a subsequence of the $\theta^n$, we may assume that $\gamma_n$ lies in $H$ and thus in $\theta^n(G)$. In particular we get $\hat{\theta}_n(G) = \gamma_n \theta^n(G) \gamma^{-1}_n = \theta^n(G)$.

The injective morphisms $\hat{\theta}_n: G \to G$ are pairwise distinct, since their images strictly embed one into the cother. Up to further extraction, we may assume that the sequence $(\hat{\theta}_n)_{n \in \N}$ is stable. But Proposition \ref{ShorteningQuotientsAreProper} for $G=\Gamma$ and $f$ the identity on $H$ tells us that it must then have non trivial stable kernel, which contradicts the fact that the morphisms $\hat{\theta}_n$ are  all injective. 
\end{proof}

To prove Proposition \ref{FactorSet}, we need
\begin{prop} \label{FiniteMaxQuotient} Let $G$ be a finitely generated group, and let $\Gamma$ be a torsion-free hyperbolic group. Suppose ${\cal S}$ is a set of stable sequences $(\theta_n)_{n \in \N}$ of homomorphisms $G \to \Gamma$ which satisfies:
\begin{itemize}
    \item ${\cal S}$ is closed under extraction of subsequences;
    \item ${\cal S}$ is closed under extraction of diagonal subsequences.
\end{itemize}
Then there exists a finite number of quotient maps $\eta_j: G \to G/K_j$ where $K_j$ is the stable kernel of a sequence which lies in ${\cal S}$, such that for any element $(\theta_n)_{n \in \N}$ of ${\cal S}$, the homomorphisms $\theta_n$ eventually factor through one of the maps $\eta_j$. 
\end{prop}
This can be proved by following the argument given in \cite{Sel1} to prove Lemma 5.4 and Lemma 5.5 where this is shown in the case where $\Gamma$ is free (see Proposition 6.21 and 6.22 in \cite{ThesisPerin} for a proof in the general torsion-free hyperbolic case).

\begin{proof}[Proof of Theorem \ref{FactorSet}] Note first that if $G$ is cyclic, any non injective morphism $G \to \Gamma$ is trivial, since $\Gamma$ is torsion-free. We can thus assume that $G$ is not cyclic.

Let ${\cal S}$ be the set of stable sequences $(\theta_n)_{n \in \N}$ of non injective short morphisms $G \to \Gamma$ which restrict to $f$ on $H$. Proposition \ref{FiniteMaxQuotient} gives us a finite set of quotients of the form $\eta: G \to G/K$, where $K$ is the stable kernel of a sequence $(\theta_n)_{n \in \N}$ which lies in ${\cal S}$. 

Assume $K$ is trivial. We extract from the sequence $(\theta_n)_{n \in \N}$ a subsequence $(\theta_{n_k})_{k \in \N}$ of pairwise distinct morphism as follows: let $n_1=1$. Suppose we have picked $\theta_{n_k}$: it is not injective, so we can pick a non trivial element $g_k$ in its kernel. Since the stable kernel of $(\theta_n)_{n \in \N}$ is trivial, there exists $n_{k+1} > n_k$ such that for all $n \geq n_{k+1}$, we have $\theta_{n}(g_k) \neq 1$. Clearly we get in this way a sequence of pairwise distinct morphisms. By Proposition \ref{ShorteningQuotientsAreProper}, $K$ is non trivial: this is a contradiction.

Let now $\theta: G \to \Gamma$ be a non injective morphism which restricts to $f$ on $H$, and let $\tau$ in $\Mod_H(G)$ and $\gamma$ in $Z(f(H))$ be such that $\hat{\theta} = \Conj(\gamma) \circ \theta \circ \tau$ is short with respect to $H$. The constant sequence $(\hat{\theta})_{n \in \N}$ lies in ${\cal S}$, so $\hat{\theta}$ eventually factors through one of the quotients of our finite set. But this means that $\theta \circ \tau$ factors through this quotient.
\end{proof}

\section{A special case} \label{SpecialCaseSec}
To give an insight of how we prove homogeneity of finitely generated free groups, we will now show the following result:

\begin{prop} Let $\F$ be a finitely generated group of rank at least $2$. Suppose $\bar{g}, \bar{ g}'$ are tuples of elements of $\F$ such that $\tp^{\F}(\bar{g}) = \tp^{\F}(\bar{g}')$. Suppose moreover that $\F$ is freely indecomposable with respect to each of the subgroups $H= \langle \bar{g} \rangle $ and $H'= \langle \bar{g}' \rangle$, and that both the JSJ decomposition of $\F$ relative to $H$ and relative to $H'$ are trivial. Then there is an automorphism of $\F$ which sends $\bar{g}$ to $\bar{g}'$.
\end{prop}

\begin{proof} The triviality of the JSJ decomposition implies that there are no non-trivial cyclic splittings of $\F$ in which $H$ or $H'$ is elliptic, hence $\Mod_H(\F)$ and $\Mod_{H'}(\F)$ are trivial.  

It is enough to show that there exist injective morphisms $i, j: \F \to \F$ with $i(\bar{g}) = \bar{g}'$ and  $j(\bar{g}')=\bar{g}$. Indeed, then $i \circ j$ is an injective automorphism of $\F$ which fixes $\bar{g}$, so by Theorem \ref{RelativeCoHopf} it is in fact an automorphism. Thus $i$ is also bijective, and we get the result.

Suppose thus by contradiction that (without loss of generality) there is no injective morphism $i: \F \to \F$ with $i(\bar{g}) = \bar{g}'$. Define the morphism $f:H \to \F$ by $f(\bar{g})= \bar{g}'$. We get by Theorem \ref{FactorSet} a finite set of proper quotients $\{ \eta_i: \F \to M_i \}_{1 \leq i \leq r}$, and since all the morphisms $\F \to \F$ sending $\bar{g}$ to $\bar{g}'$ are non injective, \textbf{any} such morphism factors through one of the maps $\eta_i$ (we do not need to precompose by anything since the modular group is trivial). We will now translate this statement as a first order sentence satisfied by $\bar{g}'$ over $\F$, and see that $\bar{g}$ cannot satisfy it: this will contradict $\bar{g}$ and $\bar{g}'$ having the same type.

Fix a basis $\bar{a}= (a_1, \ldots, a_n)$ of $\F$, and write the tuple $\bar{g}$ as a tuple of words $\bar{g}(\bar{a})$ in the basis elements $a_i$. Pick also a non trivial element $v_i$ in the kernel of each map $\eta_i$, and write it as a word $v_i(\bar{a})$. A morphism $\F \to \F$ is just a choice of image $\bar{x}=(x_1, \ldots, x_n)$ for the basis, and the image of the element represented by a word $w(\bar{a})$ is represented by $w(\bar{x})$. Thus the following first-order statement holds on $\F$:

$$\phi(\bar{g}'): \; \forall \bar{x} \; \left\{ [\bar{g}(\bar{x}) = \bar{g}' ] \Rightarrow [\bigvee^{r}_{i=1} \bar{v}_i (\bar{x})= 1 ] \right\}.$$

The statement $\phi(\bar{g})$ says that any morphism $\F \to \F$ which sends $\bar{g}$ to $\bar{g}$ kills one of the $v_i$. But the identity is an obvious counterexample. 
\end{proof}

An interesting exercise is to generalize this proof assuming not that the JSJ decomposition is trivial, but that it has no surface groups (i.e. that all the vertex groups of the JSJ decomposition are rigid). In the general case, however, the modular group does not translate well into first-order, and the difficulty comes precisely from the surface groups. 

\section{Hyperbolic towers and preretractions}   \label{HypTowersAndPreretractionsSec}
\begin{defi} \label{HypFloor} Consider a triple $(G, G', r)$ where $G$ is a group, $G'$ is a subgroup of $G$, and $r$ is a retraction from $G$ onto $G'$. We say that $(G, G', r)$ is a hyperbolic floor if there exists a non trivial decomposition $\Lambda$ of $G$ as a graph of groups with a distinguished subset $V_S$ of vertices which are of surface type such that:
\begin{itemize}
\item $G'$ is the subgroup of $G$ generated by the groups of the vertices of $\Lambda$ which are not in $V_S$, and it is in fact their free product;
\item the surfaces corresponding to vertices in $V_S$ are either punctured tori or of characteristic at most $-2$;
\item every edge $e$ of $\Lambda$ has exactly one endpoint $v_e$ in $V_S$ (bipartism), the edge group $G_e$ is a maximal boundary subgroup of the vertex group $G_{v_e}$, and this induces a bijection between edges of $\Lambda$ adjacent to $v_e$ and conjugacy classes of maximal boundary subgroups of $G_{v_e}$;
\item the retraction $r$ sends groups of vertices of $V_S$ to non abelian images. 
\end{itemize}
\end{defi}
By an abuse of language, we call the vertices of $V_S$ the surface type vertices of $\Lambda$, and other vertices the non surface type vertices (although some of them might have surface type according to Definition \ref{SurfaceTypeVertex}).

\begin{defi} \label{HypTower}
Let $G$ be a group, let $H$ be a subgroup of $G$.
We say that $G$ is a hyperbolic tower based on $H$ if there exists a finite
sequence $G=G^0 \geq G^1 \geq \ldots \geq G^m \geq H$ of subgroups of $G$ with $m \geq 0$ and:
\begin{itemize}
\item for each $k$ in $[0, m-1]$, there exists a  retraction $r_k:G^{k} \rightarrow G^{k+1}$
such that the triple $(G^k, G^{k+1}, r_k)$ is a hyperbolic floor, and $H$ is contained in one of the non surface type vertex group of the corresponding hyperbolic floor decomposition;
\item $G^m$ is not abelian and $G^m = H * F * S_1 * \ldots * S_p$ where $F$ is a (possibly trivial) free group, $p \geq 0$, and each $S_i$
is the fundamental group of a closed surface of Euler characteristic at most $-2$.
\end{itemize}
\end{defi}

\begin{rmk} \label{HypTowersAreHyp} By Bestvina and Feighn's combination theorem (see \cite{BFCombinationTheorem}), hyperbolic towers over torsion-free hyperbolic groups are themselves torsion-free hyperbolic groups. 
\end{rmk}

Hyperbolic towers were introduced by Sela in \cite{Sel1}, and he shows in \cite{Sel6} the following result:
\begin{thm}[\cite{Sel6}] \label{ElementaryEqToFree} A finitely generated group $G$ is elementary equivalent to a non abelian free group if and only if  it has a structure of hyperbolic tower over the trivial subgroup.
\end{thm}

Finitely generated free groups admit a structure of hyperbolic tower over any of their free factors. The following result states that this is in fact the only tower sructure a free group can admit over one of its subgroups.

\begin{prop} \label{NoHypFloorInFreeGroup} There is no retraction $r$ from the free group $\F_k$ to one of its subgroups which makes $(\F_k, r(\F_k), r)$ into a hyperbolic floor. 
\end{prop}
The proof of this result is contained in the proof of Theorem 1.3 in \cite{PerinElementary}.  

We now want to give a result which enables us to deduce that a group $G$ has a structure of hyperbolic floor over a proper subgroup, assuming the existence of a map $G \to G$ which preserves some of the properties of a cyclic (relative) JSJ decomposition for $G$. We need to define these maps slightly more generally as maps $U \to G$ where $U$ is a subgroup of $G$.

\begin{defi} \label{Preretraction} Let $G$ be a group, let $U$ be a subgroup of $G$, and let $\Lambda$ be a cyclic (relative) JSJ decomposition of $U$. A morphism $U \to G$ is a preretraction with respect to $\Lambda$ if its restriction to each rigid vertex group $U_v$ of $\Lambda$ is a conjugation by some element $g_v$ of $G$, and if flexible vertex groups have non abelian images. 
\end{defi} 

The following result appears as Proposition 5.11 in \cite{PerinElementary}.
\begin{prop} \label{Retraction} \label{RETRACTION} Let $U$ be a torsion-free hyperbolic group. Let $\Lambda$ be a cyclic (relative) JSJ decomposition of $U$.
Assume that there exists a non injective preretraction $U \rightarrow U$ with respect to $\Lambda$. Then there exists a subgroup $U'$ of $U$ and a retraction $r$ from $U$ to $U'$ such that $(U,U',r)$ is a hyperbolic floor. Moreover, given a rigid type vertex group $R_0$ of $\Lambda$, we can choose $U'$ to contain $R_0$.
\end{prop}

The second proposition will enable us to get, starting with a preretraction $U \to G$, a preretraction from $U$ to a subgroup $G'$ over which $G$ has a structure of hyperbolic tower. It appears as Proposition 5.12 in \cite{PerinElementary}.
\begin{prop} \label{RetractionPlus} \label{RETRACTIONPLUS} Let $G$ be a torsion-free hyperbolic group. Let $U$ be a non cyclic retract of $G$, and let $\Lambda$ be a cyclic (relative) JSJ decomposition of $U$. 
Suppose $G'$ is a subgroup of $G$ containing $U$ such that either $G'$ is a free factor of $G$, or $G'$ is a retract of $G$ by a retraction $r:G \to G'$ which makes $(G, G', r)$ a hyperbolic floor.
If there exists a non-injective preretraction $U \to G$ with respect to $\Lambda$, then there exists a non-injective preretraction $U \to G'$ with respect to $\Lambda$. 
\end{prop}

\section{Homogeneity of the free groups} \label{MainResultSec}
The aim of this section is to show that non abelian free groups are strongly $\aleph_0$-homogeneous. We start by showing that finitely generated free groups are homogeneous. 

\subsection{Finitely generated free groups}

Let $G$ and $G'$ be groups, and let $B$ and $B'$ be subgroups of $G$ and $G'$ respectively which are isomorphic via $f:B \to B'$. If $\phi$ is a first-order formula with parameters in $B$, we denote by $f(\phi)$ the formula with parameters in $B'$ obtained from $\phi$ by replacing each parameter $b$ by $f(b)$. If $p$ is a type over $B$, we denote by $f(p)$ the set of formulas with parameters in $B'$ consisting of the formulas $f(\phi)$ for $\phi \in p$. 

The following result is the key intermediate step to prove homogeneity of finitely generated free groups.
\begin{prop} \label{MainResult} Let $G$ and $G'$ be torsion-free hyperbolic groups. Let $B$ and $B'$ be subgroups of $G$ and $G'$ respectively which are isomorphic via a map $f:B \to B'$.

Let $\bar{ u}=(u_1, \ldots, u_n)$  and $\bar{ v}=(v_1, \ldots, v_n)$ be tuples of elements of $G$ and $G'$ respectively. Let $H_u$ be the subgroup generated by $\bar{u}$ and $B$, and let $U$ be a finitely presented subgroup of $G$ which contains $H_u$, and is freely indecomposable with respect to it. 

Suppose that $f(\tp^G(\bar{u}/B)) = \tp^{G'}(\bar{v}/B')$. Then either there exists an embedding $U \hookrightarrow G'$ extending $f$ which sends $\bar{ u}$ to $\bar{ v}$, or there exists a non injective preretraction $r: U \to G$ with respect to a JSJ decomposition $\Lambda$ of $U$ in which $H_u$ is elliptic. 
\end{prop}

In fact, we will deduce this from the following special case where $B$ is empty:
\begin{prop} \label{MainResultBis} Let $G$ and $G'$ be torsion-free hyperbolic groups.
Let $\bar{ u}=(u_1, \ldots, u_n)$  and $\bar{ v}=(v_1, \ldots, v_n)$ be tuples of elements of $G$ and $G'$ respectively. Let $U$ be a finitely presented subgroup of $G$ which contains $H_u=\langle \bar{u} \rangle$, and is freely indecomposable with respect to it.

Suppose that $\tp^G(\bar{ u}) = \tp^{G'}(\bar{ v})$. Then either there exists an embedding $U \hookrightarrow G'$ which sends $\bar{ u}$ to $\bar{ v}$, or there exists a non-injective preretraction $r: U \to G$ with respect to the JSJ decomposition $\Lambda$ of $U$ relative to $H_u$. 
\end{prop}

Let us first see why this implies the more general result stated in Proposition \ref{MainResult}.
\begin{proof}[Proof of Proposition \ref{MainResult}] We also denote by $f$ the extension of $f$ to $H_u$ which sends $\bar{ u}$ to $\bar{ v}$.

By Lemma \ref{ModGroupStabilizes}, there is a finitely generated subgroup $H_0$ of $H_u$ such that $U$ is freely indecomposable with respect to $H_0$, $H_u$ is elliptic in the JSJ decomposition $\Lambda$ of $U$ with respect to $H_0$, and $\Mod_{H_0}(U) = \Mod_{H_u}(U)$. Moreover, by Corollary \ref{SetOfInjStabilizes}, we may also assume $H_0$ large enough so that an embedding $U \hookrightarrow G'$ which restricts to $f$ on $H_0$ also restricts to $f$ on $H_u$.

Now if we let $\bar{h}$ be a finite generating set for $H_0$, it is easy to see that $\tp^G(\bar{h}) = \tp^{G'}(f(\bar{h}))$, so by Proposition \ref{MainResultBis}, either there is an embedding $U \hookrightarrow G'$ restricting to $f$ on $H_0$, or there is a non injective preretraction $U \hookrightarrow G$ with respect to $\Lambda$. This implies the result.
\end{proof}

To prove Proposition \ref{MainResultBis}, we need the following definition.
\begin{defi} Let $U$ be a group, let $\Lambda$ be a (relative) cyclic JSJ decomposition of $U$, and let $h$ be a morphism from $U$ to a group $G$. We say that a morphism $h':U \to G$ is $\Lambda$-related to $h$ if
\begin{itemize}
    \item for each rigid vertex group $R$ of $\Lambda$, there exists an element $u_R$ such that the restriction of $h'$ to $R$ is $\Conj(u_R) \circ h|_R$;
  \item for each flexible vertex group $S$ of $\Lambda$, if $S$ has non abelian image by $h$, it also has non abelian image by $h'$.
\end{itemize}
\end{defi}

\begin{rmk} \label{RelToiIsPreretraction}
Suppose $U$ is a subgroup of a group $G$, and let $\Lambda$ be a cyclic (relative) JSJ decomposition for $U$.
\begin{itemize}
    \item $h$ is $\Lambda$-related to an embedding $U \hookrightarrow G$ iff it is a preretraction with respect to $\Lambda$;
    \item if $\sigma$ is a (relative) modular automorphism, then by Lemma \ref{ModAndJSJ}, the map $h'=h \circ \sigma$ is $\Lambda$-related to $h$.
\end{itemize}
\end{rmk}

The following straightforward lemma, which states that $\Lambda$--relatedness can be expressed in first-order logic, is stated as Lemma 5.17 in \cite{PerinElementary}.  
\begin{lemma} Let $U$ be a group generated by a finite tuple $\bar{u}$, and let $\Lambda$ be a (relative) cyclic JSJ decomposition of $U$. There exists a first order formula $\Rel(\bar{x}, \bar{y})$ such that for any pair of morphisms $h$ and $h'$ from $U$ to $G$, the morphism $h'$ is $\Lambda$-related to $h$ if and only if $G \models \Rel(h(\bar{u}), h'(\bar{u}))$.
\end{lemma}

\begin{proof}[Proof of Proposition \ref{MainResultBis}] Denote by $f$ the homomorphism $H_u \to G'$ sending $\bar{u}$ to $\bar{v}$. 

Suppose that there is no embedding $U \hookrightarrow G'$ extending $f$. By Proposition \ref{FactorSet}, there exists a finite set of proper quotients of $U$ through which any morphism $U \to G'$ extending $f$ factors after precomposition by an element $\sigma$ of the modular group $\Mod_{H_u}(U)$.

Pick a non trivial element $g_j$ in the kernel of each of these quotients $\eta_j:U \to U_j$. Pick a finite presentation $\langle \bar{s} \mid \Sigma_U(\bar{s}) \rangle $ for $U$. The elements of the tuple $\bar{u}$ and the elements $g_j$ are represented by words $\bar{u}(\bar{s})$ and $g_j(\bar{s})$ in $\bar{s}$.

For any morphism $\theta: U \to G'$ which extends $f$, there exists a morphism $\theta': U \to G'$ (namely the morphism $\theta \circ \sigma$) which is $\Lambda$-related to it by Remark \ref{RelToiIsPreretraction}, and for which $\theta'(g_j)=1$ for some $j$. We can express this by a first order sentence $\phi(\bar{ v})$ as follows: 

$$ \forall \bar{x} \; \left\{ [\Sigma_U(\bar{x})=1 \wedge (\bar{v} = \bar{u}(\bar{x}))] \Rightarrow \; \exists \bar{y} [\Sigma_U(\bar{y})=1 \wedge \Rel(\bar{x},\bar{y}) \wedge \bigvee_j g_j(\bar{y}) = 1] \right\} . $$

We have $G' \models \phi(\bar{ v})$, so by hypothesis $G \models \phi(\bar{ u})$. The interpretation of $\phi(\bar{ u})$ gives us that for any morphism $\theta: U \to G$ which fixes $\bar{u}$, there is a non injective morphism $\theta': U \to G$ which is $\Lambda$-related to $\theta$. If we take $\theta$ to be the embedding of $U$ in $G$, by Remark \ref{RelToiIsPreretraction} we get a non injective preretraction $U \to G$ with respect to $\Lambda$. 
\end{proof}

In the particular case where $G = G'$, Proposition \ref{MainResult} implies
\begin{thm} \label{MainResult2} Let $G$ be a torsion-free hyperbolic group, and let $B$ be a subgroup of $G$. Let $\bar{ u}=(u_1, \ldots, u_n)$  and $\bar{ v}=(v_1, \ldots, v_n)$ be tuples of elements of $G$.

Suppose that $\tp^G(\bar{ u}/B) = \tp^G(\bar{ v}/B)$. Then either there is an automorphism of $G$ which restricts to the identity on $B$ and sends $\bar{ u}$ to $\bar{ v}$, or there exists a retraction $r$ on $G$ which makes $(G, r(G), r)$ a hyperbolic floor, and such that either $H_u = \langle B, \bar{u} \rangle$ or $H_v = \langle B, \bar{v} \rangle$ lies in $r(G)$. 
\end{thm}

\begin{proof} Denote by $U$ (respectively $V$) the smallest free factor of $G$ containing $H_u$ (respectively $H_v$). Any injective morphism $i:U \to G$ restricting to the identity on $B$ and sending $\bar{ u}$ to $\bar{ v}$ has image in $V$, and conversely any injective morphism $j:V \to G$ restricting to the identity on $B$ and sending $\bar{ v}$ to $\bar{ u}$ has image in $U$. 

Suppose first that such morphisms $i, j$ exist: then the morphism $j \circ i$ is an injective morphism $U \to U$ which fixes $H_u$. By Theorem \ref{RelativeCoHopf}, it is in fact an automorphism. Thus $i$ is an isomorphism $U \to V$. By Grushko's theorem, it can be extended to an isomorphism $G \to G$ since $U$ and $V$ are free factors of $G$. Since this map sends $\bar{ u}$ to $\bar{ v}$ and fixes $B$, the first alternative is satisfied.

But if (without loss of generality) there is no injective morphism $U \to G$ fixing $B$ and sending $\bar{ u}$ to $\bar{ v}$, we can apply Theorem \ref{MainResult} to get a non injective preretraction $U \to G$ with respect to the JSJ decomposition of $U$ relative to $H_u$. Note also that in this case, $U$ is not cyclic: if $U$ was cyclic generated by an element $u_0$, all the elements of $\bar{u}$ and of $B$ would be powers of $u_0$, thus by equality of the types all the elements of $\bar{v}$ and of $B$ would be powers of a common element $v_0$. Then the map $U \to G$ sending $u_0$ to $v_0$ is clearly injective, fixes $B$, and sends $\bar{ u}$ to $\bar{ v}$.

By Proposition \ref{RetractionPlus}, there exists a non injective preretraction $U \to U$. By Proposition \ref{Retraction}, we get a retraction $r: U \to r(U)$ which makes $(U, r(U), r)$ into a hyperbolic floor, and we can choose it so that the vertex group of $\Lambda$ containing $H_u$ lies in $r(U)$. This easily extends to a hyperbolic floor structure of $G$ over a proper subgroup containing $H_u$.
\end{proof}

\begin{cor} \label{FreeGroupHomogeneous} Let $\bar{u}$ and $\bar{v}$ be finite tuples of elements of a finitely generated free group $\F$, and let $B$ be a subgroup of $\F$. If $\tp^{\F}(\bar{ u}/B) = \tp^{\F}(\bar{ v}/B)$, then there exists an automorphism of $\F$ which restricts to the identity on $B$ and sends $\bar{u}$ to $\bar{v}$.
\end{cor}

\begin{proof} If there is no automorphism of $\F$ sending $\bar{ u}$ to $\bar{ v}$ and fixing $B$, then Theorem \ref{MainResult2} implies that $\F$ admits a structure of hyperbolic floor over a proper subgroup. But by Proposition \ref{NoHypFloorInFreeGroup}, no such structure exists. 
\end{proof}

\begin{rmk} Similarly, the fundamental groups of the sum of three projective planes and four projective planes respectively are homogeneous, since such groups do not admit proper hyperbolic tower structures.
\end{rmk}

\subsection{Free groups of infinite rank}

We now show that any free group of infinite rank is strongly $\aleph_0$-homogeneous. 
We begin by recalling the proof of the following classical fact of model theory (for a definition of elementary chain, see Section \ref{ModelTheorySec}):
\begin{lemma} \label{ChainOfAleph0Homogeneous}
The union of an elementary chain of $\aleph_0$-homogeneous structures is $\aleph_0$-homoge\-neous.
\end{lemma}

\begin{proof}
Let $\mathcal{M}=\bigcup_{i\in I}\mathcal{M}_i$ be the union of an elementary chain 
of $\aleph_0$-homogeneous structures. Let $\bar{a},\bar{b}$ be finite tuples of $\mathcal{M}$, 
with  $tp^{\mathcal{M}}(\bar{a})=tp^{\mathcal{M}}(\bar{b})$, 
and let $c$ be an element of $\mathcal{M}$. There is an index $j\in I$ such that $\bar{a},\bar{b},c$ are in $\mathcal{M}_j$. By the elementary chain theorem mentioned in Section \ref{ModelTheorySec}, $\mathcal{M}_j$ is an elementary substructure of $\mathcal{M}$, thus $tp^{\mathcal{M}_j}(\bar{a})=tp^{\mathcal{M}_j}(\bar{b})$. Now $\mathcal{M}_j$ is 
$\aleph_0$-homogeneous so there is $d\in\mathcal{M}_j$, such that 
$tp^{\mathcal{M}_j}(\bar{a}c)=tp^{\mathcal{M}_j}(\bar{b}d)$. This implies $tp^{\mathcal{M}}(\bar{a}c)=tp^{\mathcal{M}}(\bar{b}d)$.     
\end{proof}

\begin{prop}
Let $\kappa \geq \aleph_0$. Then $\F_{\kappa}$ is strongly $\aleph_0$-homogeneous. 
\end{prop}
\begin{proof}
As noted in Section \ref{ModelTheorySec}, $\F_{\aleph_0}$ is the union of the elementary chain $(\F_n)_{n \in \omega}$. By Theorem \ref{MainResult}, this is in fact an elementary chain of $\aleph_0$-homogeneous structures, so by Lemma \ref{ChainOfAleph0Homogeneous}, $\F_{\aleph_0}$ is $\aleph_0$-homogeneous. We pick an injection $\aleph_0 \hookrightarrow \kappa$, and see $\F_{\aleph_0}$ as a free factor of $\F_{\kappa}$ via the induced  natural embedding. Recall that by Lemma \ref{FKappaElementaryInFLambda}, this embedding is elementary in $\F_{\kappa}$.
Now let $\bar{a},\bar{b}$ be finite tuples of $\F_{\kappa}$, with $tp^{\F_{\kappa}}(\bar{a})=tp^{\F_{\kappa}}(\bar{b})$. As $\bar{a}$ is a finite tuple, we can find $f_1\in Aut(\F_{\kappa})$ such that $f_1(\bar{a})=\bar{a}^{\prime}\subset \F_{\aleph_0}$.
Similarly, we can find $f_2\in Aut(\mathbb{F})$ such that $f_2(\bar{b})=\bar{b}^{\prime}\subset F_{\aleph_0}$. 
Since $\F_{\aleph_0}$ is an elementary substructure, we have  $tp^{\F_{\aleph_0}}(\bar{a}^{\prime})=tp^{\F_{\aleph_0}}(\bar{b}^{\prime})$. 
But $\F_{\aleph_0}$ is $\aleph_0$-homogeneous and countable, therefore it is strongly $\aleph_0$-homogeneous. 
Hence, there is $f$ in $\Aut(\F_{\aleph_0})$ such that $f(\bar{a}^{\prime})=\bar{b}^{\prime}$. 
As $\F_{\aleph_0}$ is a free factor of $\F_{\kappa}$, there is $f^{\prime}\in \Aut(\F_{\kappa})$ extending $f$. 
Finally, $f_2^{-1}\circ f^{\prime}\circ f_1$ is an automorphism of $\F_{\kappa}$ which sends $\bar{a}$ to $\bar{b}$.
 
\end{proof}

\section{Characterization of elements of primitive type} \label{ElementsOfprimitiveTypeSec}
We want to understand better the type $p_0$ of a primitive element in a free group. The aim of this section is to give a characterization of elements which realize this type in a finitely generated group $G$. Note that if such an element exists, it implies in particular that $G$ is elementary equivalent to a non abelian free group (the type of an element of $G$ contains in particular all the closed formulas which are true on $G$). Thus, as noted in Remark \ref{HypTowersAreHyp}, the group $G$ must be torsion-free hyperbolic.

Let $\F_n$ be the free group on $a_1, \ldots, a_n$ for $n \geq 2$. We show the following result
\begin{prop} \label{ElementsOfPrimitiveType} Let $G$ be a  finitely generated group. Let $(u_1, \ldots, u_k)$ be a $k$-tuple of elements of $G$ for $1 \leq k \leq n$ such that $\tp^G(u_1, \ldots, u_k)= \tp^{\F_n}(a_1, \ldots, a_k)$. Then the subgroup $H_u = \langle u_1, \ldots, u_k \rangle$ is free of rank $k$, and $G$ admits a structure of hyperbolic tower over $H_u$. 
\end{prop}

\begin{rmk} The fundamental group of the connected sum of four projective planes does not admit any non trivial structure of hyperbolic tower, so it admits no structure of hyperbolic tower over a cyclic subgroup. This implies that the generic type is not realised in this model of the theory of free groups.
\end{rmk}

To prove Proposition \ref{ElementsOfPrimitiveType}, we need the following definition:
\begin{defi} A hyperbolic subtower of a finitely generated group $G$ is a subgroup $G'$ of $G$ such that $G$ has a structure of hyperbolic tower over $G'$. If $G_1$ and $G_2$  are hyperbolic subtowers of $G$, we say $G_2$ is smaller than $G_1$ if $G_1$ admits a structure of hyperbolic tower over $G_2$ (so in particular $G_2$ is a subgroup of $G_1$).
\end{defi}

\begin{lemma} If $G$ is torsion-free hyperbolic, for any subset $A$ of $G$ there exists a minimal hyperbolic subtower containing $A$.
\end{lemma}

Note that the minimal hyperbolic subtower containing $\{1\}$ is precisely Sela's elementary core, defined in \cite[Definition 7.5]{Sel7}.

\begin{proof} Any subtower of $G$ is finitely generated, since it is a retract of $G$. As a finitely generated subgroup of $G$, it is a $G$-limit group. By Theorem 1.12 in \cite{Sel7}, an infinite sequence of epimorphisms between $G$-limit groups stabilizes. 
\end{proof}

\begin{proof}[Proof of Proposition \ref{ElementsOfPrimitiveType}] It is easy to see that $H_u$ is free of rank $k$. 

Let $U$ be a minimal subtower of $G$ containing $H_u$. If $U=H_u$, the result holds, so let us assume this is not the case. Note that, as a minimal subtower containing $H_u$, the group $U$ is freely indecomposable with respect to $H_u$. Suppose $j$ is an injective morphism $U \to \F$ sending $u_i$ to $a_i$: it must send $U$ into $\langle a_1, \ldots, a_k \rangle$ since this is the minimal free factor of $\F$ containing $j(H_u)$. But $j(H_u) = \langle a_1, \ldots, a_k \rangle$ which contradicts injectivity of $j$. Thus there are no injective morphisms $j:U \to \F$ sending $u_i$ to $a_i$.

As noted above, $G$ is torsion-free hyperbolic, so we can apply Theorem \ref{MainResult}. Thus there exists a non-injective preretraction $U \to G$ with respect to a JSJ decomposition $\Lambda$ of $U$ in which $H_u$ is elliptic. By Proposition \ref{RetractionPlus}, since $G$ is a hyperbolic tower over $U$, there exists a non injective preretraction $U \to U$ with respect to $\Lambda$. Finally, applying Proposition \ref{Retraction}, we get that there is a retraction $r$ from $U$ to a proper subgroup $U'$ containing $H_u$ such that $(U, U', r)$ is a hyperbolic floor. This contradicts the minimality of $U$.
\end{proof}

In fact, we will show that the converse of Theorem \ref{ElementsOfPrimitiveType} holds, so that we have
\begin{thm} \label{ElementsOfPrimitiveTypeIFF}Let $G$ be a finitely generated group. Let $(u_1, \ldots, u_k)$ be a $k$-tuple of elements of $G$ for $k \leq n$. Then $\tp^G(u_1, \ldots, u_k)= \tp^{\F_n}(a_1, \ldots, a_k)$ if and only if the subgroup $H_u= \langle u_1, \ldots, u_k \rangle$ is free of rank $k$ and $G$ admits a structure of hyperbolic tower over $U$.
\end{thm}

For this, we will need to use
\begin{thm}[\cite{SelaPrivate}] \label{ConverseSela}  Let $G$ be a torsion-free hyperbolic group. If $G$ has a structure of hyperbolic tower over a non abelian subgroup $H$, then $H$ is elementarily embedded in $G$.
\end{thm}

The proof of this result has not appeared in print, but follows exactly the proof of Theorem 7.6 in \cite{Sel7}, where Sela shows that in a torsion-free hyperbolic group $G$ which is not elementary equivalent to the free group, the embedding of a minimal hyperbolic subtower over $\{1\}$ (an "elementary core" of $G$) is elementary in $G$. The hypothesis that $G$ is not elementary equivalent to the free group only serves to deduce that such a minimal hyperbolic subtower is not trivial, and the minimality hypothesis is not used. 

\begin{proof}[Proof of Theorem \ref{ElementsOfPrimitiveTypeIFF}] Proposition \ref{ElementsOfPrimitiveType} proves one direction of this equivalence. To see the other, we suppose first that $k \geq 2$. The group $H_u = \langle u_1, \ldots, u_k \rangle$ is not abelian, so by Theorem \ref{ConverseSela}, it is elementarily embedded in $G$. Since it is free of rank $k$, we have $\tp^U(u_1, \ldots, u_k)=\tp^{\F_k}(a_1, \ldots, a_k)$ where $\F_k$ is the free factor of $\F_n$ generated by $a_1, \ldots, a_k$. Now by Theorem \ref{ConverseSela} again, this free factor is elementary in $\F_n$ so $\tp^{\F_k}(a_1, \ldots, a_k)= \tp^{\F_n}(a_1, \ldots, a_k)$. This proves the result for $k \geq 2$.

Suppose now $k=1$. Let $G_0$ be the bottom floor of the hyperbolic tower structure $G$ admits over $\langle u \rangle$. Note that $G_0$ is in particular non abelian.  By Theorem \ref{ConverseSela}, the embedding of $G_0$ in $G$ is elementary, so it is enough to show that $\tp^{G_0}(u) = \tp^{\F_n}(a_1)$. We assume without loss of generality that $G_0$ is minimal in the following sense: it does not admit a hyperbolic floor structure over a proper subgroup $G'_0$ containing $u$.

Now $G_0$ is the free product of $\langle u \rangle$ with possibly some closed surface groups of characteristic at most $-2$, and possibly a free group. Note that if a closed surface group of characteristic at most $-2$ is not the fundamental group of the connected sum of four projective planes, it has a non-trivial structure of hyperbolic tower. On the other hand, if $S = \langle a, b, c, d \mid a^2 b^2 = c^2 d^2 \rangle$ is the fundamental group of the connected sum of four projective planes, then the morphism $r$ from $\langle u \rangle * S$ to $\langle u \rangle * \langle a \rangle$ given by $r(u)=u$, $r(a)=a$, $r(c)=a$ and $r(b)=r(d) = u$ is a retraction which makes $(\langle u \rangle * S, \langle u, a \rangle, r)$ into a hyperbolic floor. Thus no closed surface group can appear in the free product decomposition of $G_0$ as it would contradict its minimality.

But now, $G_0$ has a free factor of the form $\langle u \rangle*\Z$, which by Theorem \ref{ConverseSela} is elementarily embedded in $G_0$ (and thus in $G$). Thus the type of $u$ in $G_0$ is the same as the type of $u$ in $\langle u \rangle * \Z$, which is precisely $\tp^{\F_n}(a_1)$.
\end{proof}

\section{Surface groups} \label{SurfaceCaseSec}
Let us now consider the fundamental group $S$ of a closed hyperbolic surface $\Sigma$ which is not the connected sum of three projective planes. By Theorem \ref{ElementaryEqToFree}, the group $S$ is elementary equivalent to the non abelian free group $\F$ over $a_1, \ldots, a_n$. We want to find elements whose type in $S$ is the same as the type $p_0$ of $a_1$ in $\F$. We already remarked that there are no such elements if $S$ is the connected sum of four connected planes, we can thus exclude this case as well.

Theorem \ref{ElementsOfPrimitiveTypeIFF} tells us $S$ must have a structure of hyperbolic tower over the subgroup generated by such an element. Let us recall the description of the non trivial hyperbolic towers structures $S$ admits given in \cite{PerinElementary}. Minor subsurfaces are defined by
\begin{defi} Let $\Sigma$ be a closed surface. Let $\Sigma_0$ be a proper connected subsurface of $\Sigma$. Denote by $\Sigma_1$ the closure of $\Sigma \setminus \Sigma_0$. We say that $\Sigma_0$ is a minor subsurface of $\Sigma$ if 
\begin{enumerate}
	\item $\Sigma_1$ is connected;
  \item $\chi(\Sigma_0) \geq \chi(\Sigma)/2$, with equality if and only if $\Sigma_0$ and $\Sigma_1$ are homeomorphic; 
	\item if $\Sigma_0$ is non orientable, so is $\Sigma_1$.
\end{enumerate}
\end{defi}

The following is stated as Theorem 5.20 in \cite{PerinElementary}:
\begin{prop} Let $S$ be the fundamental group of a closed hyperbolic surface $\Sigma$ which is not the connected sum of three or four projective planes. Then $S$ admits a structure of hyperbolic tower over a proper subgroup $H$ if and only if $H$ is a free factor of the fundamental group of a minor subsurface of $\Sigma$.
\end{prop}

In particular, if $S_0$ is the fundamental group of a minor subsurface of $\Sigma$, by Theorem \ref{ConverseSela} it is an elementary subgroup of $S$ and $\tp^{S_0}(u) = \tp^{S}(u)$ for any element $u$ of $S_0$. Thus any primitive element $u$ of the free group $S_0$ is such that $\tp^{S}(u) = \tp^{\F}(a_1)$. 

If $u$ is an element of $S$ which represents a non separating simple closed curve $\gamma$ on $\Sigma$, it is easy to see that we can find a minor subsurface $\Sigma_0$ of $\Sigma$ containing $\gamma$ such that $u$ is primitive in $S_0 = \pi_1(\Sigma_0)$. Thus elements representing non separating simple closed curves on $\Sigma$ have the same type in $S$ as primitive elements in $\F_n$.

\begin{rmk} If $u$ represents a separating simple closed curve, it can be written as a non trivial product of commutators and squares, so its type is not $\tp^{\F}(a_1)$.
\end{rmk}

If $\Sigma$ is non orientable, it admits non separating simple closed curves which are $1$-sided and some which are $2$-sided. The elements representing each of these types of curves both have the type of $a_1$ in $\F$, but they are not in the same orbit under $\Aut(S)$. Thus, we see that non orientable surfaces are not homogeneous. To see this for all surfaces, we will have to consider also elements which do not represent simple closed curves. Denote by ${\cal C}(\Sigma)$ the set of elements of $S$ which represent simple closed curves of $\Sigma$.

\begin{prop} \label{SCCNotAType} Let $S$ be the fundamental group of a closed hyperbolic surface $\Sigma$ which is not the connected sum of three or four projective planes. There exists an element $u$ of $S - {\cal C}(\Sigma)$ such that $\tp^{S}(u) = \tp^{\F}(a_1)$.
\end{prop}

Since $\Aut(S)$ preserves ${\cal C}(\Sigma)$, Proposition \ref{SCCNotAType} implies
\begin{cor}  \label{SurfaceNotHomogeneous} Let $S$ be the fundamental group of a closed hyperbolic surface $\Sigma$ which is not the connected sum of four projective planes. Then $S$ is not homogeneous.
\end{cor}

To prove Proposition \ref{SCCNotAType}, we need the following result. It appears as Remark 7.2 in \cite{BirmanSeriesSCC} (see also \cite{RivinSCC} and \cite{MirzakhaniGrowthSCC}).
\begin{thm} \label{GrowthGeodesicSCC} Let $\Sigma$ be a closed hyperbolic surface. There exists a polynomial $P_{\Sigma}$ such that the number of closed simple geodesic curves on $\Sigma$ of length at most $L$ is bounded above by $P_{\Sigma}(L)$.
\end{thm}

From this we get:
\begin{cor} \label{GrowthOfSimpleClosedCurves} Let $S$ be the fundamental group of a closed hyperbolic surface $\Sigma$. Fix a finite generating set for $S$, denote by $B_n$ the set of elements of $S$ represented by words of length at most $n$. Then the growth function $f_{{\cal C}(\Sigma)}: \N \to \N$ given by $f_{{\cal C}(\Sigma)}(n) = |B_n \cap {\cal C}(\Sigma)|$ is bounded by a polynomial in $n$. 
\end{cor}
This follows from Theorem \ref{GrowthGeodesicSCC} using the fact that the isotopy class of a simple closed curve on $\Sigma$ contains exactly one geodesic representative, and the $S$-equivariant quasi-isometry between the Cayley graph of $S$ and the universal cover of $\Sigma$.

\begin{proof}[Proof of Proposition \ref{SCCNotAType}] Let $S_0$ be the fundamental group of a minor subsurface of $\Sigma$. We saw that any primitive element $u$ of $S_0$ is such that $\tp^S(u) = \tp^{\F}(a_1)$. But the set of primitive elements of $S_0$ grows exponentially in $S_0$, and thus in $S$ as well. The set ${\cal C}(\Sigma)$ grows at most polynomially in $S$ by Theorem \ref{GrowthOfSimpleClosedCurves}. This proves the result. 
\end{proof}

Note that Proposition \ref{SCCNotAType} also implies:
\begin{cor} Let $S$ be the fundamental group of a closed hyperbolic surface $\Sigma$ which is not the connected sum of three or four projective planes. The set ${\cal C}(\Sigma)$ is not definable over the empty set.
\end{cor}

We showed that two elements representing non separating simple closed curves on a surface have the same type, and they are in the same orbit if and only if the corresponding curves are both one-sided, or both two-sided. To complete the picture regarding the types of simple closed curves, we prove

\begin{prop} Let $S$ be the fundamental group of a closed hyperbolic surface $\Sigma$ which is not the connected sum of three or four projective planes. Elements representing separating simple closed curves have the same type if and only if they are in the same orbit under $\Aut(S)$. 
\end{prop}

\begin{proof}Let $c$ and $d$ be two elements of $S$ representing separating simple closed curves $\gamma$ and $\delta$ on $\Sigma$. Suppose that $\tp^S(c)=\tp^S(d)$. Denote by $\Sigma^1_c$ and $\Sigma^2_c$ (respectively $\Sigma^1_d$ and $\Sigma^2_d$) the connected components of $\Sigma - \gamma$ (respectively $\Sigma - \delta$). We denote the fundamental group of $\Sigma^i_{\epsilon}$ by $S^i_{\epsilon}$ for $i \in \{1,2\}$ and $\epsilon \in \{c,d\}$. We define the complexity of a surface with boundary $\Theta$ to be $k(\Theta)=(r, -n)$, where $r$ is the rank of its fundamental group as a free group and $n$ is the number of its boundary components, and we order complexities lexicographically.

Assume without loss of generality that $k(\Sigma^1_c) = \min_{i, \epsilon} \{k(\Sigma^i_{\epsilon})\}$. If $\Sigma^1_c$ is orientable, $S^1_c$ admits a presentation of the form $\langle c, a_1, \ldots a_{2r} \mid c= \prod_j [a_{2j-1}, a_{2j}] \rangle$. The formula $ \exists x_1, \ldots x_{2r} \; \; c= \prod_j [x_{2j}, x_{2j+1}] $ is in the type of $c$ in $S$, hence it is also satisfied by $d$. This gives us a morphism $f:S^1_c \to S$ sending $c$ to $d$. We can proceed in the same way for the non orientable case. 

Denote by $\Lambda_d$ the graph of group decomposition of $S$ dual to $\delta$, it corresponds to an amalgamated product of the form $S^1_d *_{\langle d \rangle} S^2_d$. Note that since the centralizer of $d$ in $S$ is exactly $\langle d \rangle$, the tree corresponding to $\Lambda_d$ is $1$-acylindrical, that is, distinct edges have stabilizers which intersect trivally. 

Suppose first that the kernel of the map $f$ does not contain any elements which represent simple closed curves on $\Sigma^1_c$. The surface group with boundary $S^1_c$ acts on the tree corresponding to $\Lambda_d$ via the map $f$, and boundary elements of $S^1_c$ are elliptic in this action. By Theorem III.2.6 of \cite{MorganShalen}, the splitting induced on $S^1_c$ by this action is dual to a non empty set of disjoint simple closed curves on $\Sigma^1_c$, and its vertex groups are fundamental groups of proper subsurfaces of $\Sigma^1_c$. Denote by $\hat{S}^1_c$ the fundamental group of one of these subsurface $\hat{\Sigma}^1_c$, and suppose without loss of generality that $f(\hat{S}^1_c) \leq S^1_d$. The map $f$ sends boundary elements of $\hat{S}^1_c$ to boundary elements of $S^1_d$, so by Lemma 3.10 in \cite{PerinElementary}, $f(\hat{S}^1_c)$ has finite index in $S^1_d$. By Lemma 3.12 in \cite{PerinElementary}, this implies that the complexity of $\hat{\Sigma}^1_c$, and thus that of $\Sigma^1_c$, is at least that of $\Sigma^1_d$, and if we have equality $f$ is an isomorphism $S^1_c \to S^1_d$. By minimality of $k(\Sigma^1_c)$, we do have equality. 

If the kernel of the map $f$ does contain elements representing simple closed curve on $\Sigma^1_c$, we proceed as in Section 5.7 of \cite{PerinElementary}. We pick an essential set of curves killed by $f$, and define the pinching map $\rho_{\cal C}$ 
and the graph of groups $\Gamma(S^1_c, {\cal C})$. 
The map $f$ then factors as $f' \circ \rho_{\cal C}$, and the kernel of $f'$ does not contain elements which represent simple closed curves on the surfaces of $\Gamma(S^1_c, {\cal C})$, including on the unique surface $\hat{\Sigma}^1_c$ whose fundamental group contains $\rho_{\cal C}(d)$. By repeating the argument above for the map $f'$ and the surface $\hat{\Sigma}^1_c$, we see that the complexity of $\hat{\Sigma}^1_c$ is at least that of $\Sigma^1_d$. But $k(\hat{\Sigma}^1_c) < k(\Sigma^1_c)$, this is a contradiction.  

Thus we have an isomorphism $f: S^1_c \to S^1_d$ sending $c$ to $d$. To finish the proof, it is enough to find an isomorphism $h: S^2_c \to S^2_d$ sending $c$ to $d$. Since $\chi(\Sigma^1_c) + \chi(\Sigma^2_c) = \chi(\Sigma)= \chi(\Sigma^1_d) + \chi(\Sigma^2_d)$, and all the surfaces involved have one boundary component, we see that we must have $k(\Sigma^2_c) = k(\Sigma^2_d)$. If $\Sigma^2_c$ and $\Sigma^2_d$ have the same orientability, we are done. Thus let us assume that $\Sigma^2_c$ is orientable and $\Sigma^2_c$ isn't. In this case, note that $\Sigma$ is non orientable so $\Sigma^1_c$ and $\Sigma^1_d$ must be non orientable.

As we did for $S^1_c$, we can find (using equality of the types of $c$ and $d$) a morphism $h: S^2_d \to S$ sending $c$ to $d$. Again we start by assuming no elements representing simple closed curves on $\Sigma^2_c$ are in the kernel of $h$. If $h(S^2_c)$ is not elliptic, by acylindricity of $\Lambda_d$ we can find a proper subsurface of $\Sigma^2_c$ whose fundamental group is sent by $h$ to a conjugate of $S^2_d$. But then its complexity must be at least that of $\Sigma^2_d$, while being strictly smaller than that of $\Sigma^2_c$, this is a contradiction. Thus $h(S^2_c)$ is elliptic in $\Lambda_d$, i.e. $h(S^2_c) \leq S^j_d$ for $j \in \{1,2\}$. If $\Ker(h)$ does contain elements representing simple closed curves, as before we can apply this argument to a surface obtained by pinching to get a contradiction.

The morphism $h$ sends non conjugate maximal boundary elements of $S^2_c$ to non conjugate maximal boundary elements of $S^j_d$, so by Lemma 5.22 of \cite{PerinElementary}, it is an isomorphism. But both $\Sigma^1_d$ and $\Sigma^2_d$ are non orientable, which contradicts orientability of $\Sigma^2_c$. 
\end{proof}

%
%


%

\bibliography{biblio}

\providecommand{\bysame}{\leavevmode\hbox to3em{\hrulefill}\thinspace}
\providecommand{\MR}{\relax\ifhmode\unskip\space\fi MR }
\providecommand{\MRhref}[2]{%
  \href{http://www.ams.org/mathscinet-getitem?mr=#1}{#2}
}
\providecommand{\href}[2]{#2}
\begin{thebibliography}{GL09b}

\bibitem[BF92]{BFCombinationTheorem}
Mladen Bestvina and Mark Feighn, \emph{{A combination theorem for negatively
  curved groups}}, J. Differential Geom. \textbf{35} (1992), 85--101.

\bibitem[Bow98]{Bowditch}
Brian Bowditch, \emph{{Cut points and canonical splittings of hyperbolic
  groups}}, Acta Math. \textbf{180} (1998), 145--186.

\bibitem[BS85]{BirmanSeriesSCC}
Joan Birman and Caroline Series, \emph{Geodesics with bounded intersection
  number on surfaces are sparsely distributed.}, Topology \textbf{24} (1985),
  217--225.

\bibitem[Cha]{Chatzidakis}
Zoé Chatzidakis, \emph{{Introduction~to~model~theory}}, Notes, available at
  \texttt{http://www.logique.jussieu.fr/~zoe/index.html}.

\bibitem[CK90]{ChangKiesler}
C.C. Chang and H.J. Kiesler, \emph{Model theory}, North-Holland Publishing Co.,
  1990.

\bibitem[DS99]{DunwoodySageev}
Martin Dunwoody and Michah Sageev, \emph{{JSJ-splittings for finitely presented
  groups over slender groups}}, Invent. Math. \textbf{135} (1999), 25--44.

\bibitem[FP06]{FujiwaraPapasoglu}
Koji Fujiwara and Panos Papasoglu, \emph{{JSJ-decompositions of finitely
  presented groups and complexes of groups}}, Geom. Funt. Anal \textbf{16}
  (2006), 70--125.

\bibitem[GL08]{GuirardelLevittTreeOfCylinders}
Vincent Guirardel and Gilbert Levitt, \emph{{Tree of cylinders and canonical
  splittings}}, 2008.

\bibitem[GL09a]{GuirardelLevittJSJI}
\bysame, \emph{{JSJ decompositions: definitions, existence, uniqueness. I: The
  JSJ deformation space}}, 2009.

\bibitem[GL09b]{GuirardelLevittJSJII}
\bysame, \emph{{JSJ decompositions: definitions, existence, uniqueness. II:
  Compatibility and acylindricity}}, 2009.

\bibitem[Gui08]{GuirardelRTrees}
Vincent Guirardel, \emph{{Actions of finitely generated groups on $\R$-trees}},
  Ann. Inst. Fourier \textbf{58} (2008), 159--211.

\bibitem[KM06]{KharlampovichMyasnikov}
Olga Kharlampovich and Alexei Myasnikov, \emph{{Elementary theory of free
  nonabelian groups}}, J. Algebra \textbf{302} (2006), 451--552.

\bibitem[Mar02]{MarkerModelTheory}
David Marker, \emph{Model theory: an introduction}, Graduate Texts in
  Mathematics, vol. 217, Springer, 2002.

\bibitem[Mir08]{MirzakhaniGrowthSCC}
Maryam Mirzakhani, \emph{Growth of the number of simple closed geodesics on a
  hyperbolic surface}, Ann. of Math. \textbf{168} (2008), 97--125.

\bibitem[MS84]{MorganShalen}
John~W. Morgan and Peter~B. Shalen, \emph{{Valuations, trees, and degenerations
  of hyperbolic structures I}}, Ann. of Math. \textbf{120} (1984), 401--476.

\bibitem[Nie03]{NiesF2}
Andre Nies, \emph{Aspects of free groups}, J. Algebra \textbf{263} (2003),
  119--125.

\bibitem[NS90]{NadelStavi}
M.~Nadel and J.~Stavi, \emph{On models of the elementary theory of $(\z,+,1)$},
  Journal of Symbolic Logic (1990), 1--20.

\bibitem[Per08]{ThesisPerin}
Chlo\'e Perin, \emph{Elementary embeddings in torsion-free hyperbolic groups},
  Ph.D. thesis, Universit\'e de Caen Basse-Normandie, October 2008.

\bibitem[Per09]{PerinElementary}
Chloé Perin, \emph{Elementary embeddings in torsion free hyperbolic groups},
  arXiv:0903.0945v1 [math.GR], 2009.

\bibitem[Pil96]{PillayStability}
Anand Pillay, \emph{Geometric stability theory}, Oxford University Press, 1996.

\bibitem[Pil08]{PillayForking}
\bysame, \emph{Forking in the free group}, J. Inst. Math. Jussieu \textbf{7}
  (2008), 375--389.

\bibitem[Pil09]{PillayGenericity}
\bysame, \emph{On genericity and weight in the free group}, Proc. Amer. Math.
  Soc. \textbf{137} (2009), 3911--3917.

\bibitem[Poi01]{PoizatStableGroups}
Bruno Poizat, \emph{Stable groups}, Mathematical Surveys and Monographs,
  vol.~87, AMS, 2001.

\bibitem[Riv01]{RivinSCC}
Igor Rivin, \emph{Simple curves on surfaces}, Geom. Dedicata \textbf{87}
  (2001), 345--360.

\bibitem[RS94]{RipsSelaHypI}
Eliyahu Rips and Zlil Sela, \emph{{Structure and rigidity in hyperbolic groups
  I}}, Geom. Funct. Anal. \textbf{4} (1994), 337--371.

\bibitem[RS97]{RipsSelaJSJ}
\bysame, \emph{{Cyclic splittings of finitely presented groups and the
  canonical JSJ decomposition}}, Ann. of Math. \textbf{146} (1997), 53--104.

\bibitem[Sela]{SelaPrivate}
Zlil Sela, {Private communication}.

\bibitem[Selb]{Sel7}
\bysame, \emph{{Diophantine geometry over groups VII: The elementary theory of
  a hyperbolic group}}, Proceedings of the LMS, {to appear}.

\bibitem[Selc]{SelaStability}
\bysame, \emph{{Diophantine geometry over groups VIII: Stability}}, {preprint,
  available at \texttt{http://www.ma.huji.ac.il/~zlil/}}.

\bibitem[Sel97]{SelaHypII}
\bysame, \emph{{Structure and rigidity in (Gromov) hyperbolic groups and
  discrete groups in rank 1 Lie groups II}}, Geom. Funct. Anal. \textbf{7}
  (1997), 561--593.

\bibitem[Sel01]{Sel1}
\bysame, \emph{{Diophantine geometry over groups {$\rm I$}. Makanin-Razborov
  diagrams}}, Publ. Math. Inst. Hautes Études Sci. \textbf{93} (2001), 31--105.

\bibitem[Sel06]{Sel6}
\bysame, \emph{{Diophantine geometry over groups VI: The elementary theory of
  free groups}}, Geom. Funct. Anal. \textbf{16} (2006), 707--730.

\bibitem[Ser83]{SerreTrees}
Jean-Pierre Serre, \emph{{Arbres, amalgames, $SL_2$}}, Astérisque \textbf{46}
  (1983).

\bibitem[Skl]{SklinosGenericType}
Rizos Sklinos, \emph{On the generic type of the free group}, to appear.

\bibitem[Wil06]{WiltonThesis}
Henry Wilton, \emph{Subgroup separability of limit groups}, Ph.D. thesis,
  Imperial College, London, 2006, available at
  \texttt{http://www.math.utexas.edu/users/henry.wilton/thesis.pdf}.

\end{thebibliography}

\end{document}